\documentclass[b5paper]{article}
[12pt]

\makeatletter
\renewcommand*\l@section{\@dottedtocline{1}{1.5em}{2.3em}}
\makeatother

\usepackage{amsfonts}
\usepackage{amssymb}
\usepackage[T1]{fontenc}

\usepackage{tikz}
\usetikzlibrary{calc}

\usepackage{CJK}
\usepackage{amsmath}

\usepackage{amsfonts}
\usepackage{amssymb}
\usepackage{amsthm}
\usepackage{amssymb}
\usepackage{enumerate}
\usepackage[calc]{picture}
\usepackage[all,cmtip]{xy}

\usepackage[mathscr]{eucal}
\usepackage{eqlist}

\usepackage{color}
\usepackage{abstract}
\usepackage[T1]{fontenc}

 \usepackage[top=1.5 cm, bottom=1.8cm, left=2.8cm, right= 2.8cm]{geometry}

\usepackage{cite}

\theoremstyle{plain}
\newtheorem{theorem}{Theorem}[section]
\newtheorem{proposition}[theorem]{Proposition}
\newtheorem{lemma}[theorem]{Lemma}
\newtheorem{corollary}[theorem]{Corollary}

\theoremstyle{definition}
\newtheorem{definition}{Definition}[section]

\theoremstyle{example}
\newtheorem{example}{Example}[section]

 \theoremstyle{remark}%{myrem}

\usepackage{etoolbox}

\numberwithin{equation}{section}
\numberwithin{theorem}{section}

\begin{document}

\title{ {\Large  Generations       of  random hypergraphs  and   random  simplicial  complexes  by  the  map  algebra}}
\author{Shiquan Ren }

\date{}

\maketitle

\begin{abstract}
We  consider  the  random  hypergraph  on  a  finite vertex  set  by  choosing each  set of  vertices  as  an   hyperedge  independently  at  random.   We  express  the  probability  distributions  of  
 the  (lower-)associated  simplicial  complex    and  the  (lower-)associated  independence  hypergraph   of   
the  random  hypergraph    in  terms  of  the   probability  distributions  of  certain  random  simplicial complex   and  certain  random  independence hypergraph   of  Erd\"os-R\'enyi  type.    We  construct  a  graded  structure   of  the  map  algebra   explicitly  and   give    algorithms  to  generate  random  hypergraphs  and  random  simplicial  complexes.  
\end{abstract}



\section{Introduction}

Let  $V$  be  a  finite  vertex  set.   Let  $\Delta[V]$  be  the collection  of all  the  nonempty  subsets  of  $V$.  
  Let  $p:  \Delta[V]\longrightarrow  [0,1]$  be  a   function  on  $\Delta[V]$  with values  in the unit  interval.   A   hypergraph   on  $V$  is  a  subset of  $\Delta[V]$.  
A  hypergraph  is  an  (abstract)   simplicial complex   if     any   nonempty  subset    of   any hyperedge   is  still a  hyperedge  in  the hypergraph.  
For  any  simplicial complex  $\mathcal{K}$  on  $V$,  an  external face of  $\mathcal{K}$  is  a  nonempty  subset $\sigma$  of $V$  such that $\sigma\notin \mathcal{K}$  and  $\tau\in \mathcal{K}$  for any proper  subset  
$\tau$  of  $\sigma$.   Let  $E(\mathcal{K})$  be  the set  of  all the  external faces of $\mathcal{K}$.    We  call a  hypergraph  an  independence  hypergraph   if     any    superset,   which  is  a  subset  of   $V$,       of   any hyperedge   is  still a  hyperedge  in  the hypergraph.   The  complement  of   a  simplicial  complex  in  $\Delta[V]$  is  an  independence  hypergraph  and  vice versa.    For  any  independence  hypergraph   $\mathcal{L}$  on  $V$,  a   co-external face of  $\mathcal{L}$  is  a     subset $\sigma$  of $V$  such that $\sigma\notin \mathcal{L}$  and  $\tau\in \mathcal{L}$  for any proper  superset   
$\tau$  of  $\sigma$  such  that  $\tau$ a  subset of  $V$.   Let  $\bar  E(\mathcal{L})$  be  the set  of  all the  co-external faces of $\mathcal{L}$.

Let   $\mathcal{H}$   be   a  hypergraph  on  $V$.  The  associated  simplicial complex $\Delta\mathcal{H}$  of  $\mathcal{H}$
is the smallest  simplicial  complex    containing  $\mathcal{H}$  (cf.   \cite{bfchen, parks, h1}).
 The  lower-associated  simplicial  complex  $\delta\mathcal{H}$  of  $\mathcal{H}$  is  the  largest  simplicial  complex    contained in  $\mathcal{H}$  (cf.  \cite{ph}).   We  define  the  associated  independence   hypergraph   $\bar\Delta\mathcal{H}$  of  $\mathcal{H}$  
as the smallest   independence  hypergraph   containing  $\mathcal{H}$   and  define     
 the  lower-associated  independence  hypergraph   $\bar\delta\mathcal{H}$  of  $\mathcal{H}$  as  the  largest  independence  hypergraph   contained in  $\mathcal{H}$.    Consider  the  random  hypergraph  whose  probability   is  given by 
  \begin{eqnarray*}
  \bar{\rm P}_{p}(\mathcal{H})=\prod_{\sigma\in\mathcal{H}}  p(\sigma)\prod_{\sigma\notin   \mathcal{H}}\big(1-p(\sigma)\big),
  \end{eqnarray*}
  where   $ \mathcal{H} $   is  any  hypergraph  on  $V$,      
  the       random  simplicial  complex  whose  probability   is  given by 
  \begin{eqnarray*}
  {\rm P}_{p}(\mathcal{K})=\prod_{\sigma\in\mathcal{K}}  p(\sigma)\prod_{\sigma\in  E(\mathcal{K})}\big(1-p(\sigma)\big),
  \end{eqnarray*}
  where   $ \mathcal{K} $   is  any  simplicial  complex  on  $V$,       
 and  the      random  independence  hypergraph  whose  probability   is  given by 
    \begin{eqnarray*}
{\rm Q}_{p}(\mathcal{L})=\prod_{\sigma\in\mathcal{L}}  p(\sigma)\prod_{\sigma\in  \bar  E(\mathcal{L})}\big(1-p(\sigma)\big),
  \end{eqnarray*}
  where   $ \mathcal{L} $   is  any  independence  hypergraph  on  $V$.   Note that  $\bar{\rm P}_{p}$,  
  $ {\rm P}_{p}$  and  ${\rm Q}_{p}$  are  well-defined  probability  functions,  i.e. 
  \begin{eqnarray*}
  \sum_{ \mathcal{H} }  \bar{\rm P}_{p}(\mathcal{H})=
    \sum_{ \mathcal{K} }  {\rm P}_{p}(\mathcal{K})=
      \sum_{ \mathcal{L} }   {\rm Q}_{p}(\mathcal{L})=1.  
  \end{eqnarray*}
  The  summations  are      over  all  the  hypergraphs  on  $V$,  all  the  simplicial  complexes  on  $V$  and  
  all  the  independence  hypergraphs  on  $V$  respectively.  %  It will be  seen  in  Subsection~\ref{sss3.22222}  that    the    random  models 
 % $\bar{\rm P}_{p}$,  ${\rm P}_{p}$  and  ${\rm  Q}_{p}$  
 % are  the  probabilities  of   randomly generated  of    Erd\"os-R\'enyi-type.    

\begin{theorem}\label{th-mmmmm11111}
Let  $\mathcal{H}\sim  \bar{\rm  P}_p$   be  a  randomly  generated  hypergraph  on  $V$.  
Then  
\begin{enumerate}[(1).]
\item
the  complement   of   $\mathcal{H}$  is  a randomly  generated  hypergraph  $\gamma\mathcal{H}\sim  \bar{\rm  P}_{1-p}$,  
\item
the  complement  of  the  associated  simplicial  complex  of   $\mathcal{H}$   is  a  randomly  generated  
independence  hypergraph  $\gamma \Delta\mathcal{H} \sim  {\rm  Q}_{1-p}$,  

\item
the  complement  of  the  associated  independence  hypergraph   of   $\mathcal{H}$   is a randomly generated  
simplicial  complex    $\gamma \bar\Delta\mathcal{H} \sim  {\rm   P}_{1-p}$,  

\item
the  lower-associated  simplicial  complex   of   $\mathcal{H}$   is   a  randomly generated  
simplicial  complex   $\delta\mathcal{H} \sim  {\rm  P}_{p}$, 
\item
the  lower-associated  independence  hypergraph   of   $\mathcal{H}$     is   a   randomly generated  
 independence  hypergraph   $ \bar\delta\mathcal{H}\sim  {\rm  Q}_{p}$. 
\end{enumerate}
\end{theorem}

We  will  prove  Theorem~\ref{th-mmmmm11111}  in   Section~\ref{s3333333a}.   In   fact,   Theorem~\ref{th-mmmmm11111}    (1)  follows  from  Lemma~\ref{cplmt},   Theorem~\ref{th-mmmmm11111}   (2),  (3)  follow  from  
 Corollary~\ref{co-1y3058}~(1),  (2)  respectively,   and    Theorem~\ref{th-mmmmm11111}   (4),   (5)   follow  from  
 Theorem~\ref{le-3.3aa}~(3),  (4)  respectively.   As  by-products,  we  have  the following  two corollaries.

 \begin{corollary}\label{co-23ho2hr}
 Let  $V'$  and  $V''$   be   two  disjoint  vertex  sets.   Let  $p':  \Delta[V']\longrightarrow  [0,1]$  and   let  $p'':  \Delta[V'']\longrightarrow  [0,1]$.    Define  $p'*p'':  \Delta[V'\sqcup  V'']\longrightarrow  [0,1]$  by  letting  $(p'*p'')(\sigma)=p'(\sigma\cap V') p''(\sigma\cap V'')$  for any $\sigma\in \Delta[V'\sqcup V'']$.  
 \begin{enumerate}[(1).]
 \item
 Let  $\mathcal{H}'\sim  \bar{\rm  P}_{p'} $  and    $\mathcal{H}''\sim  \bar{\rm  P}_{p''} $  be  randomly  generated  hypergraphs.    Then their  join  is  a  randomly  generated  hypergraph  $\mathcal{H}'*\mathcal{H}''\sim  \bar{\rm  P}_{p'*p''}$;   
  \item
   Let  $\mathcal{K}'\sim   {\rm  P}_{p'} $  and    $\mathcal{K}''\sim  {\rm  P}_{p''} $  be  randomly  generated   simplicial  complexes.    Then their  join  is  a  randomly  generated  simplicial  complex  $\mathcal{K}'*\mathcal{K}''\sim   {\rm  P}_{p'*p''}$; 
  \item
   Let  $\mathcal{L}'\sim   {\rm  Q}_{p'} $  and    $\mathcal{L}''\sim  {\rm  Q}_{p''} $  be  randomly  generated    independence  hypergraphs.    Then their  join  is  a  randomly  generated  independence  hypergraph    $\mathcal{L}'*\mathcal{L}''\sim   {\rm   Q}_{p'*p''}$.  
 \end{enumerate}
 \end{corollary}
 
 Corollary~\ref{co-23ho2hr}~(1)  follows  from     Lemma~\ref{pr-88.a8}   and     Corollary~\ref{co-23ho2hr}~(2),   
 (3)   follow   from  Corollary~\ref{co-1l3n5rdvg}~(1),   (2)   respectively.
 
 \begin{corollary}\label{co-2xt.log}
 Let  $p',  p'':  \Delta[V]\longrightarrow  [0,1]$.   Define  
 $p'\cap  p'',   p'\cup  p'':  \Delta[V]\longrightarrow  [0,1]$  by  letting  
   $(p'\cap p'')(\sigma)= p'(\sigma) p''(\sigma)$  and   $(p'\cup p'')(\sigma)= 1-(1-p'(\sigma))(1- p''(\sigma))$  for any $\sigma\in\Delta[V]$.
   \begin{enumerate}[(1).]
 \item
 Let  $\mathcal{H}'\sim  \bar{\rm  P}_{p'} $  and   let  $\mathcal{H}''\sim  \bar{\rm  P}_{p''} $.    Then their  intersection  is  a  randomly  generated  hypergraph  $\mathcal{H}'\cap \mathcal{H}''\sim  \bar{\rm  P}_{p'\cap  p''}$   and  their   union   is  a  randomly  generated  hypergraph  $\mathcal{H}'\cup \mathcal{H}''\sim  \bar{\rm  P}_{p'\cup  p''}$;   
  \item
   Let  $\mathcal{K}'\sim   {\rm  P}_{p'} $  and  let  $\mathcal{K}''\sim  {\rm  P}_{p''} $.    Then their  intersection   is  a  randomly  generated  simplicial  complex  $\mathcal{K}'\cap \mathcal{K}''\sim   {\rm  P}_{p'\cap   p''}$; 
  \item
   Let  $\mathcal{L}'\sim   {\rm  Q}_{p'} $  and   let  $\mathcal{L}''\sim  {\rm  Q}_{p''} $.    Then their  intersection  is  a  randomly  generated  independence  hypergraph    $\mathcal{L}'\cap\mathcal{L}''\sim   {\rm   Q}_{p'\cap   p''}$.  
 \end{enumerate}
 \end{corollary}

Corollary~\ref{co-2xt.log}~(1)  follows  from  \cite[Lemma~4.5]{jktr2}  (Lemma~\ref{le-3.1aa})   and   Corollary~\ref{co-2xt.log}~(2),   (3)    follow   from  Corollary~\ref{co-nwtllip}~(1),   (2)  respectively.

\smallskip

The  generation  of  random  hypergraphs  as  well  as  random  simplicial  complexes  is  useful  in  computer  science    (for  example,  \cite{naheed}).     The  map  algebra on  the  space  of   random  sub-hypergraphs   of   a  fixed  simplicial  complex  was  initially  studied  by  
C.  Wu,  J.  Wu   and  the  present author    \cite{jktr2}.  In Section~\ref{ss8ss88}  of  this paper,  we  consider  the  map  algebra  on   the  space  of   random  sub-hypergraphs  of  the  complete  complex   $\Delta[V]$.  By  considering  the  compositions  of   $\Delta$,  $\bar\Delta$,  
$\delta$,  $\bar\delta$,  $\gamma$,  the  intersections,  the  unions,  the  joins    and    the  products,   we   construct  a    graded   structure  of   the   map  algebra   in  (\ref{eq-4.jnuoh}).   We  consider  the  action  of  the  map  algebra  on  the  random hypergraph   $\mathcal{H}\sim  \bar{\rm  P}_p$,   the  random  simplicial complex   $\mathcal{K}\sim {\rm  P}_p$  and   the  random  independence  hypergraph   $\mathcal{L}\sim {\rm  Q}_p$   in  (\ref{eq-9823u592}).   Using    the  graded  structure  in    (\ref{eq-9823u592}),    we  give  some  algorithms  generating        random  hypergraphs,  random  simplicial  complexes  and  random  independence  hypergraphs.     Each  element  in  (\ref{eq-9823u592})   will  randomly  generate  a  hypergraph  and   each   element   of   certain  particular  forms   will  generate  a  random  simplicial  complex  or  a   random  independence  hypergraph.

  \subsection{Literature  review  on  random  simplicial  complexes}
  
  Let $V$  be a    vertex set  of  cardinality  $N$.   
      Denote  $\Delta[V]$  as  $\Delta_N$.  
   Let  ${\rm  sk}_r(\Delta_N)$  be  the  $r$-skeleton  of  $\Delta_N$,    $0\leq  r\leq   N-1$.
  Consider the   space  $\Omega_N^r$  consisting  of  all sub-complexes  $\mathcal{K}$   of  ${\rm  sk}_r(\Delta_N)$.  Consider   the  probability  function
   ${\rm  P}_{r,N,\mathbf{p}}: \Omega_N^r\longrightarrow \mathbb{R}$  given by
   \begin{eqnarray*}
   {\rm  P}_{r,N,\mathbf{p}}(\mathcal{K})=  \prod_{\sigma\in \mathcal{K}}  p(\sigma) \prod_{\sigma\in  E( \mathcal{K})} \big(1-p(\sigma)\big),
   \end{eqnarray*}
where  $\mathbf{p}=(p_0,p_1,\ldots,p_r)$,   $0\leq p_0,p_1,\ldots,p_r\leq 1$  are  constants,  and  $p(\sigma)= p_i$  for any   $i$-simplex  $\sigma\in  {\rm  sk}_r(\Delta_N)$,   $0\leq i\leq  r$.  This    model
${\rm  P}_{r,N,\mathbf{p}}$    was  given  by  A.  Costa  and  M.  Farber  in
\cite{m88,m1,m4}.
The  connectivity,  the  fundamental  group,  the dimension,  and the  Betti number  of  the
random  simplicial complex with  the  probability   ${\rm  P}_{r,N,\mathbf{p}}$
have  been   studied  by  A.  Costa  and  M.  Farber  in  \cite{m88,m1,m2,m3,m4}.
In  particular,  let     $0\leq  p\leq 1$   be  a  constant.
 The followings are special cases  of  ${\rm  P}_{r,N,\mathbf{p}}$:

 (1).  
   ${\rm  P}_{1,N,\mathbf{p}}$  with  $\mathbf{p}=(1,p)$  is
   the  Erd\"os-R\'enyi model $G(N,p)$.   
P.  Erd\"os  and  A.  R\'enyi  \cite{1959er}   and   E.N.  Gilbert   \cite{1959g}
constructed  the  Erd\"os-R\'enyi  model   $G(N,p)$    by choosing  each pair of  vertices  in $V$
 as  an edge independently  at  random  with probability  $p$.
  Thresholds  for  the  connectivity  of
 $G(N,p)$  were  proved  by  P.  Erd\"os  and  A.  R\'enyi    in  \cite{1960er}.

   (2).   
     ${\rm  P}_{2,N,\mathbf{p}}$  with  $\mathbf{p}=(1,1,p)$
    is  the  Linial-Meshulam  model  $Y_2(N,p)$.   
  N.  Linial and  R.  Meshulam  \cite{y2-3} constructed a  random  $2$-complex  $Y_2(N,p)$
  by  taking the  complete graph on $V$  as  the   $1$-skeleton and then  choosing each  $2$-simplex
  independently  at  random  with  probability  $p$.    The  fundamental  group,  the  homology groups,
   and  the  asphericity  as well as the hyperbolicity  of  $Y_2(N,p)$  were respectively  studied  in  \cite{6},
    \cite{y2-4,y2-5},
   and \cite{y2-1,y2-2}.

   (3).  
    ${\rm  P}_{d,N,\mathbf{p}}$  with  $\mathbf{p}=(1,\ldots,1,p)$  %,   where  the  first  $d$ coordinates  are  $1$,  the  $(d+1)$-th   coordinate  is  $p$    and   the  last  $N-d-1$  coordinates  are  $0$,   
    is  the  Meshulam-Wallach  model  $Y_d(N,p)$.   
    R.  Meshulam  and    N.  Wallach   \cite{yd-5}   constructed a  random  $d$-complex  $Y_d(N,p)$
    by  taking the $(d-1)$-skeleton  of  the  complete  complex  on  $V$  and then  choosing  each
    $d$-simplex  independently  at  random  with  probability  $p$.
    The  (co)homology groups,   the phase transition of the homology groups,   the   eigenvalues of  the  Laplacian,  the  collapsibility   and  the  topological minor
    were  respectively  studied  in  \cite{8,9,yd-1,7},  \cite{annals},  \cite{yd-2},  \cite{yd-4,8}  and    \cite{yd-3}.

   (4).  
   ${\rm  P}_{N-1,N,\mathbf{p}}$  with  $\mathbf{p}=(1,p,1,\ldots,1)$   is  the  random flag complex  (random clique complex)  $X(N,p)$  of  $G(N,p)$.   
   This   random  simplicial  complex     was  studied     by  M. Kahle  in \cite{clique, kahleann}  and  A. Costa, M. Farber and D. Horak  in \cite{cfh}.   Sharp vanishing thresholds for the cohomology of  $X(N,p)$  were  proved  by  M. Kahle  in \cite{kahleann}.

   (5).   
   ${\rm  P}_{2,N,\mathbf{p}}$  with  $\mathbf{p}=(p_0,p_1,p_2)$  was  considered  by  M. Farber  and  T. Nowik  in  \cite{f2021}.    
The multi-parameter threshold for the property that every $2$-dimensional
simplicial complex   admits a topological embedding into the  ${\rm  P}_{2,N,\mathbf{p}}$-generated random   $2$-complex asymptotically almost surely was  established.

 \section{Simplicial models  for   hypergraphs}\label{sect2}

\subsection{Hypergraphs   and   simplicial complexes}\label{sect2.1111111}

Let  $V$  be a  finite set with a total order $\prec$.  The elements of $V$ are called {\it vertices}.  Let $2^V$  be  the  power set  of  $V$.
Let  $\Delta[V]=2^V\setminus\{\emptyset\}$.
For any nonnegative  integer $n$  and any distinct vertices $v_0,v_1,\ldots, v_n\in V$    such  that  $v_0\prec v_1\prec\cdots \prec  v_n$,   we  call the set  $\sigma 
 =\{v_0,v_1,\ldots, v_n\}$  
   an {\it $n$-hyperedge} on $V$.
Let  $\sigma  
 $  be  an  $n$-hyperedge  on  $V$.
 The  {\it subset closure}  $\Delta\sigma$   of $\sigma$  is an $n$-dimensional simplicial complex whose     set  of   simplices consists of all the non-empty subsets of $\sigma$.
The  {\it  superset  closure}  $\bar\Delta\sigma$  of $\sigma$  is  $(\Delta[V]\setminus \Delta\sigma)\cup\sigma$,  in other words,  $\bar\Delta\sigma$  consists  of  all the supersets  $\tau\supseteq \sigma$  such that  $\tau\subseteq  V$.
 Note that
  $\Delta\sigma$ does not  depend on  the choice  of  $V$    while   $\bar\Delta \sigma$  depends  on  the choice  of  $V$.

\begin{definition}\cite{berge,parks,h1}
\label{d-1.2}
A  {\it hypergraph} $\mathcal{H}$  on $V$  is a collection  of  hyperedges on $V$.  In particular,
\begin{enumerate}[(1).]
\item
     $\mathcal{H}$   is an {\it (abstract)  simplicial complex}  if $\tau\in\mathcal{H}$  for any $\sigma\in\mathcal{H}$  and any nonempty subset $\tau\subseteq \sigma$.    A hyperedge  of a simplicial complex is  called a  {\it simplex}  (cf. \cite{hatcher,eat});
 \item
      $\mathcal{H}$   is called  an {\it  independence  hypergraph}  if $\tau\in\mathcal{H}$  for any $\sigma\in\mathcal{H}$  and any  superset $\tau\supseteq \sigma$ with $\tau\subseteq  V$.
 \end{enumerate}
 \end{definition}
 Let $\mathcal{H}$  be a hypergraph  on $V$.
 \begin{definition}\label{d-1.008}
 The  {\it  complement}  of  $\mathcal{H}$  is  a  hypergraph  on  $V$  given by
 \begin{eqnarray*}
 \gamma\mathcal{H}=\{\sigma\in\Delta[V]\mid  \sigma\notin \mathcal{H}\}.
 \end{eqnarray*}
 \end{definition}
 By  Definition~\ref{d-1.008},     $\gamma\mathcal{H}$   is  a  simplicial  complex
 (resp.   an   independence  hypergraph)   iff   $\mathcal{H}$    is   an   independence  hypergraph  (resp.   a simplicial  complex).      Note  that  $\gamma^2={\rm  id}$.  

 \begin{definition}
 \label{d-1.3}
The  {\it associated simplicial complex}  (cf. \cite{bfchen, parks, h1})   is
\begin{eqnarray*}\label{eq-2.1aaz}
\Delta\mathcal{H}=\bigcup_{\sigma\in \mathcal{H}} \Delta\sigma,
\end{eqnarray*}
which    is the  smallest  simplicial  complex  containing    $\mathcal{H}$.   
The  {\it lower-associated simplicial complex}  is   (cf.  \cite{ph})
\begin{eqnarray*}\label{eq-2.1aay}
\delta\mathcal{H}=\bigcup_{\Delta\sigma\subseteq\mathcal{H}} \{\sigma\},
\end{eqnarray*}
   which    is the  largest   simplicial  complex  contained in    $\mathcal{H}$.  
  The  {\it associated  independence hypergraph}  is
 \begin{eqnarray*}\label{eq-2.1aax}
 \bar\Delta\mathcal{H}=\bigcup_{\sigma\in \mathcal{H}} \bar\Delta\sigma,
 \end{eqnarray*}
 which  is the  smallest   independence hypergraph  containing    $\mathcal{H}$.
 The {\it  lower-associated  independence hypergraph}  is
 \begin{eqnarray*}\label{eq-2.1aaw}
 \bar\delta\mathcal{H}= \bigcup_{\bar\Delta\sigma\subseteq\mathcal{H}} \{\sigma\},
 \end{eqnarray*}
  which    is the  largest   independence  hypergraph   contained in    $\mathcal{H}$.
 \end{definition}
 By  Definition~\ref{d-1.3},   we  have  a  diagram
    \begin{eqnarray*}\label{dg-a1}
  \xymatrix{
  \delta\mathcal{H}\ar[rd]^-{i_\delta}  &  & \Delta\mathcal{H}\\
  & \mathcal{H} \ar[ru]^-{i_\Delta} \ar[rd]^-{i_{\bar\Delta}}&\\
  \bar\delta\mathcal{H}\ar[ru]^-{i_{\bar\delta}}  &  &\bar\Delta\mathcal{H}
  }
  \end{eqnarray*}
  such that each arrow is a  canonical inclusion  of hypergraphs.   
  \begin{lemma}\label{le-882890}
  $\bar\Delta=\gamma\delta\gamma$  and     $\bar\delta=\gamma\Delta\gamma$.  
  \end{lemma}
  \begin{proof}
  For  any    hypergraph   $\mathcal{H}$    on    $V$,    we  have
\begin{eqnarray*}
\gamma \bar\Delta \gamma \mathcal{H}
&=&\{\sigma\in\Delta[V]  \mid \sigma\notin \bar\Delta   \gamma    \mathcal{H}\}
\nonumber\\
&=&\{\sigma\in\Delta[V]  \mid \sigma\not\supseteq  \tau {\rm~for~any~}\tau\in  \gamma    \mathcal{H} \}
\nonumber\\
&=&\{\sigma\in\Delta[V]  \mid \sigma\not\supseteq  \tau {\rm~for~any~}\tau\notin  \mathcal{H},  {\rm~where~}  \tau\in\Delta[ V]\}
\nonumber\\
&=&\{\sigma\in\Delta[V]  \mid \tau\in \mathcal{H}{\rm~for~any~} \tau\subseteq\sigma,  {\rm~where~}  \tau\in\Delta[ V]\}
\nonumber\\
&=&\delta\mathcal{H}. 
\end{eqnarray*}
Therefore,  $\bar\Delta=\gamma\delta\gamma$.    Moreover,    For  any    hypergraph   $\mathcal{H}$    on    $V$,    we  have 
\begin{eqnarray*}
\gamma    \Delta\gamma    \mathcal{H}&=&\{\sigma\in\Delta[V]  \mid \sigma\notin \Delta\gamma    \mathcal{H}\}
\nonumber\\
&=&\{\sigma\in\Delta[V]  \mid \sigma\not\subseteq  \tau {\rm~for~any~}\tau\in  \gamma    \mathcal{H}\}
\nonumber\\
&=&\{\sigma\in\Delta[V]  \mid \sigma\not\subseteq  \tau {\rm~for~any~}\tau\notin  \mathcal{H},  {\rm~where~}\tau\in\Delta[V]  \}
\nonumber\\
&=&\{\sigma\in\Delta[V]  \mid \tau\in \mathcal{H}{\rm~for~any~}\tau\supseteq\sigma,  {\rm~where~}\tau\in\Delta[V]  \}
\nonumber\\
%&=&\{\sigma\in\Delta[V]  \mid\sigma\in \bar\delta \mathcal{H} \}\\
&=&\bar\delta    \mathcal{H}.
\end{eqnarray*}
Therefore,   $\bar\delta=\gamma\Delta\gamma$.  
  \end{proof}
  
   By  \cite[Subsection~2.1]{jktr2}  and  Lemma~\ref{le-882890},   
  the  following  relations  among  $\gamma$,  
  $\Delta$,  $\delta$,  $\bar\Delta$  and  $\bar\delta$  hold:     
 (1).  $\gamma^2={\rm  id}$,  (2).  $\bar\Delta=\gamma\delta\gamma$  and     $\bar\delta=\gamma\Delta\gamma$,   (3). $\Delta\delta=\delta$, $\bar\Delta\bar\delta=\bar\delta$, 
$\delta\Delta=\Delta$   and   $\bar\delta\bar\Delta=\bar\Delta$,   
(4).    $\Delta^2=\Delta$,  $\bar\Delta^2=\bar\Delta$,   
  $\delta^2=\delta$   and   $\bar\delta^2=\bar\delta$,      
  (5).        $(\delta\bar\Delta)^2=(\delta\gamma\delta\gamma)^2=\delta\gamma\delta\gamma=\delta\bar\Delta$,   
  $(\bar\Delta\delta)^2=(\gamma\delta\gamma\delta)^2=\gamma\delta\gamma\delta=\bar\Delta \delta$,  
  $(\Delta\bar\delta)^2=(\Delta\gamma\Delta\gamma)^2=\Delta\gamma\Delta\gamma=\Delta\bar\delta$   and    $(\bar\delta \Delta)^2=(\gamma\Delta\gamma\Delta)^2=\gamma\Delta\gamma\Delta=\bar\delta \Delta$.   Moreover,    since  the  associated  simplicial  complex  of   any  nonempty   independence  hypergraph     is    $\Delta[V]$,   we  have    
  \begin{eqnarray*}
  \Delta\bar\Delta(\mathcal{H})&=&
  \begin{cases}
   \Delta[V],  & \mathcal{H}\neq \emptyset,\\
   \emptyset,  &\mathcal{H}=\emptyset,
   \end{cases}\\
   \Delta\bar\delta(\mathcal{H})&=&
   \begin{cases}
   \Delta[V], &   \{V\}\in  \mathcal{H},\\
   \emptyset,& \{V\}\notin  \mathcal{H}. 
 \end{cases}
  \end{eqnarray*}
  Thus  
   the  following relations  hold  as well:   (6).        $\Delta^2\bar\Delta=\bar\Delta\Delta\bar\Delta=\delta\Delta\bar\Delta=\bar\delta\Delta\bar\Delta=\Delta\bar\Delta$  and   $\Delta^2\bar\delta=\bar\Delta\Delta\bar\delta=\delta\Delta\bar\delta=\bar\delta\Delta\bar\delta=\Delta\bar\delta$.

Let $\mathcal{H}$  and  $\mathcal{H}'$  be
two  hypergraphs on $V$.  We  have the following  observations:
\begin{enumerate}[(i).]
\item
$\gamma(\mathcal{H}\cap\mathcal{H}')=\gamma\mathcal{H}\cup\gamma\mathcal{H}'$,  $\gamma(\mathcal{H}\cup\mathcal{H}')=\gamma\mathcal{H}\cap\gamma\mathcal{H}'$,
\item
$\Delta(\mathcal{H} \cap\mathcal{H}')\subseteq \Delta\mathcal{H} \cap \Delta\mathcal{H}'$,
 $\Delta(\mathcal{H} \cup\mathcal{H}')= \Delta\mathcal{H} \cup \Delta\mathcal{H}'$,
 \item
$\delta(\mathcal{H} \cap\mathcal{H}')= \delta\mathcal{H} \cap \delta\mathcal{H}'$,
$\delta(\mathcal{H} \cup\mathcal{H}')\supseteq  \delta\mathcal{H} \cup \delta\mathcal{H}'$,
 \item
 $\bar\Delta(\mathcal{H} \cap\mathcal{H}')\subseteq \bar\Delta\mathcal{H} \cap \bar\Delta\mathcal{H}'$,
  $\bar\Delta(\mathcal{H} \cup\mathcal{H}')= \bar\Delta\mathcal{H} \cup \bar\Delta\mathcal{H}'$,
  \item
 $\bar\delta(\mathcal{H} \cap\mathcal{H}')= \bar\delta\mathcal{H} \cap \bar\delta\mathcal{H}'$,
 $\bar\delta(\mathcal{H} \cup\mathcal{H}')\supseteq  \bar\delta\mathcal{H} \cup \bar\delta\mathcal{H}'$.
 \end{enumerate}
If  both  $\mathcal{H}$     and    $\mathcal{H}'$   are  simplicial complexes  (resp.  independence   hypergraphs),  then   both  $\mathcal{H}\cap\mathcal{H}'$   and  $\mathcal{H}\cup\mathcal{H}'$   are  simplicial complexes (resp.  independence hypergraphs).

Let $\sigma=\{v_0,v_1,\ldots,v_n\}$  be an $n$-hyperedge on $V$  and  let   $\tau=\{u_0,u_1,\ldots,u_m\}$  be an $m$-hyperedge  on $V'$.  The  {\it box  product}  $\sigma\square \tau$  is  an  $(mn+m+n+1)$-hyperedge
\begin{eqnarray}\label{eq-88.998cc}
\sigma\square  \tau  = \{(v,u)\mid  v\in  \sigma,  u\in \tau\}
\end{eqnarray}
on    the Cartesian product     $V\times  V'$.    The  right-hand  side  of  (\ref{eq-88.998cc})   is  up  to  a
rearrangement  of   the  vertices with respect to  the  lexicographic order.   In  addition,   suppose $V\cap  V'=\emptyset$.  The {\it join}  $\sigma*\tau$   is  an $(n+m+1)$-hyperedge
\begin{eqnarray}\label{eq-88.765vs}
\sigma*\tau =\{v_0,v_1,\ldots,v_n,u_0,u_1,\ldots,u_m\}
\end{eqnarray}
on  the  disjoint  union   $V\sqcup  V'$,
where the right-hand side   of  (\ref{eq-88.765vs})  is   with respect to the total order    on  $V\sqcup V'$  given by  $v\prec  v'$  for any  $v\in V$  and  any  $v'\in  V'$.     We  observe
\begin{eqnarray*}
\Delta(\sigma  *\tau)&=&(\Delta  \sigma)  *  (\Delta  \tau),\\
 \bar\Delta(\sigma  *\tau)&=&(\bar\Delta  \sigma)  *  (\bar\Delta  \tau), 
\end{eqnarray*}
where  the  associated  independence  hypergraphs  of   $\sigma  *\tau$,    $\sigma $    and   $\tau$
   are    taken  with  respect to  $V\sqcup  V'$,      $V$   and   $V'$   respectively.

\begin{definition}\label{d-1.7}
Let  $V$  and  $V'$  be  two  finite sets.   Let $\mathcal{H}$  be  a  hypergraph  on  $V$   and   let   $\mathcal{H}'$  be a   hypergraph  on  $V'$.     
\begin{enumerate}[(1).]
\item
%Let   $V\times V'$  be  the  Cartesian  product.  
  We  define  the  {\it  box  product}   of   $\mathcal{H}$   and    $\mathcal{H}'$   as
\begin{eqnarray*}
 \mathcal{H}\square  \mathcal{H}' =\{\sigma\square \sigma'\mid \sigma\in \mathcal{H},  \sigma'\in\mathcal{H}'\}.
\end{eqnarray*}
Then  $\mathcal{H}\square  \mathcal{H}'$   is  a  hypergraph  on  $V\times V'$;
\item
%Let  $V\sqcup  V'$  be  the  disjoint  union.
 Suppose  in  addition      $V\cap V'=\emptyset$.    
  We define the {\it join}   of   $\mathcal{H}$   and    $\mathcal{H}'$      as
\begin{eqnarray*}
 \mathcal{H}*\mathcal{H}' =\{\sigma*\sigma'\mid \sigma\in\mathcal{H}{\rm ~and~}\sigma'\in\mathcal{H}'\}  \cup\mathcal{H}\cup\mathcal{H}'.
\end{eqnarray*}
Then  $\mathcal{H}*\mathcal{H}'$   is   a   hypergraph  on    $V\sqcup  V'$.
\end{enumerate}
\end{definition}

By Definition~\ref{d-1.7},   we  observe  the followings:
\begin{enumerate}[(i)'.]
\item
$\Delta(\mathcal{H} *\mathcal{H}')= \Delta\mathcal{H} * \Delta\mathcal{H}'$,
\item
$\delta(\mathcal{H} *\mathcal{H}')= \delta\mathcal{H} * \delta\mathcal{H}'$,
\item
 $\bar\Delta(\mathcal{H} *\mathcal{H}')= \bar\Delta\mathcal{H} * \bar\Delta\mathcal{H}'$,
 \item
 $\bar\delta(\mathcal{H} *\mathcal{H}')= \bar\delta\mathcal{H} * \bar\delta\mathcal{H}'$,
 \item
 $\mathcal{H}_1  *(\mathcal{H}_2  \cup \mathcal{H}_3)  =  (\mathcal{H}_1  *\mathcal{H}_2)  \cup  (\mathcal{H}_1  *\mathcal{H}_3) $,
 \item
   $\mathcal{H}_1  *(\mathcal{H}_2  \cap \mathcal{H}_3)  =  (\mathcal{H}_1  *\mathcal{H}_2)  \cap  (\mathcal{H}_1  *\mathcal{H}_3) $,
   \item
   $\mathcal{H}_1  \square (\mathcal{H}_2  \cup \mathcal{H}_3)  =  (\mathcal{H}_1  \square \mathcal{H}_2)  \cup  (\mathcal{H}_1  \square \mathcal{H}_3) $,
   \item
   $\mathcal{H}_1  \square (\mathcal{H}_2  \cap \mathcal{H}_3)  =  (\mathcal{H}_1  \square \mathcal{H}_2)  \cap  (\mathcal{H}_1  \square \mathcal{H}_3) $,
   \item
   $\mathcal{H}_1  \square (\mathcal{H}_2  * \mathcal{H}_3)=   (\mathcal{H}_1  \square \mathcal{H}_2)  *    (\mathcal{H}_1  \square \mathcal{H}_3) $,
   \end{enumerate}
   where     $\bar\Delta$ and $\bar\delta$   on the left-hand sides are with respect to $V\sqcup  V'$  while       $\bar\Delta$ and $\bar\delta$   on the right-hand sides are with respect to  $V$  and  $V'$.

Let $\mathcal{K}$  be   a  simplicial  complex  on  $V$  and   let  $\mathcal{K}'$  be a   simplicial complex  on $V'$.  
 Then
  $\mathcal{K}*\mathcal{K}'$  is  a   simplicial complex   on  $V\sqcup  V'$   and
  both  $\Delta(\mathcal{K}\square \mathcal{K}')$     and     $\delta(\mathcal{K}\square \mathcal{K}')$    are      simplicial  complexes   on   $V\times  V'$.   
Let $\mathcal{L}$  be  an  independence  hypergraph  on  $V$  and   let  $\mathcal{L}'$  be an  independence  hypergraph   on $V'$.   
 Then
  $\mathcal{L}*\mathcal{L}'$  is an   independence  hypergraph  on  $V\sqcup  V'$    and  
  both  $\bar\Delta(\mathcal{L}\square \mathcal{L}')$     and     $\bar\delta(\mathcal{L}\square \mathcal{L}')$    are        independence  hypergraphs   on   $V\times  V'$.

\begin{example}
Let $V=\{v_0,v_1\}$  and   let   $V'=\{v'_0,v'_1,v'_2,v'_3\}$.
Let $\mathcal{H}=\{\{v_0\}, \{v_0,v_1\} \}$  be  a  hypergraph  on  $V$  and  let  $\mathcal{H}'=\{\{v'_0,v'_1\}, \{v'_0,v'_1,v'_2\}\}$  be  a  hypergraph  on  $V'$.
Let  $v_0\prec  v_1\prec  v'_0\prec  v'_1\prec  v'_2\prec  v'_3$.   Then
\begin{eqnarray*}
\mathcal{H}*\mathcal{H}'&=&\{\{v_0\},\{v_0,v_1\},\{v'_0,v'_1\},\{v'_0,v'_1,v'_2\},\{v_0,v'_0,v'_1\},\\
&&
\{v_0,v'_0,v'_1,v'_2\},\{v_0,v_1,v'_0,v'_1\}, \{v_0,v_1,v'_0,v'_1,v'_2\}\}
\end{eqnarray*}
is  a  hypergraph  on  $V\sqcup  V'$
 and
 \begin{eqnarray*}
\mathcal{H}\square \mathcal{H}'&=&\{\{(v_0,v'_0),  (v_0,v'_1)\},
\{(v_0,v'_0),(v_0,v'_1),(v_0,v'_2)\},  \\
&& \{(v_0,v'_0),  (v_0,v'_1), (v_1,v'_0),  (v_1,v'_1)\}, \\
&& \{(v_0,v'_0),(v_0,v'_1),(v_0,v'_2),(v_1,v'_0),(v_1,v'_1),(v_1,v'_2)\}\}
\end{eqnarray*}
is  a  hypergraph  on  $V\times  V'$.
Moreover,
\begin{enumerate}[(1). ]
\item
$\Delta\mathcal{H}=\{\{v_0\},\{v_1\}, \{v_0,v_1\} \}$,    
 $\delta\mathcal{H}= \{ \{v_0\}\}$,    
 $\bar \Delta\mathcal{H}=\bar\delta\mathcal{H}=\mathcal{H}$,  
 %\item
% $\bar \delta\mathcal{H}=\emptyset$,
\item
$\Delta\mathcal{H}'=\{\{v'_0\},\{v'_1\}, \{v'_2\}, \{v'_0,v'_1\}, \{v'_0,v'_2\}, \{v'_1,v'_2\}, \{v'_0,v'_1,v'_2\}\}$,   
$ \delta\mathcal{H}'=\emptyset$,   
 $\bar \Delta\mathcal{H}'=\{\{v'_0,v'_1\},  \{v'_0,v'_1,v'_2\},  \{v'_0,v'_1,v'_2,v'_3\}\}$,    
 $\bar \delta\mathcal{H}'=\emptyset$,
\item
$\Delta\mathcal{H} *\Delta\mathcal{H}' =\Delta\{v_0,v_1,v'_0,v'_1,v'_2\}$,   
$\delta\mathcal{H}*\delta\mathcal{H}'=\{ \{v_0\}\}$,   
  $\bar\Delta\mathcal{H} *\bar\Delta\mathcal{H}' =    \{
  \{v_0, v'_0,v'_1\}, 
  \{v_0, v'_0,v'_1,v'_2\},   \{v_0, v'_0, v'_1,v'_2,v'_3\},   \{v_0,v_1,v'_0,v'_1\},\\
    \{v_0,v_1, v'_0,v'_1,v'_2\}, 
   \{v_0,v_1,v'_0,v'_1,v'_2,v'_3\}\} \cup  \mathcal{H}  \cup\mathcal{H}' $,    
  $\bar\delta\mathcal{H} *\bar\delta\mathcal{H}' =\mathcal{H}$.  
  \end{enumerate}
\end{example}

\subsection{Morphisms  of  hypergraphs  and   simplicial  maps}

Let  $V$  and  $V'$  be  two finite sets.
Let  $\mathcal{H}$  be  a  hypergraph  on  $V$  and  let  $\mathcal{H}'$  be a    hypergraph  on  $V'$.    A  {\it  morphism}  $f: \mathcal{H}\longrightarrow \mathcal{H}'$   of  hypergraphs  is a map $f: V\longrightarrow V'$  such that  for any hyperedge  $\sigma=\{v_0,v_1,\ldots,v_n\}$  in $\mathcal{H}$,  its image $f(\sigma)$  is a hyperedge  in  $\mathcal{H}'$   spanned by the (not necessarily distinct)  vertices  $f(v_0)$, $f(v_1)$, $\ldots$, $f(v_n)$.
Let   $\mathcal{K}$  be  a    simplicial  complex   on  $V$  and  let   $\mathcal{K}'$  be a      simplicial complex   on  $V'$.    We  call a  morphism  $f: \mathcal{K}\longrightarrow \mathcal{K}'$   a  {\it  simplicial map}.       Let  $\mathcal{L}$  be  an   independence  hypergraph   on  $V$  and  let   $\mathcal{L}'$  be an       independence  hypergraph   on  $V'$.    We  call a  morphism  $f: \mathcal{L}\longrightarrow \mathcal{L}'$   a  {\it  morphism}  of  independence hypergraphs.

Let  $f: \mathcal{H}\longrightarrow \mathcal{H}'$  be a morphism of  hypergraphs.
  Then   $f$  canonically  induces two  simplicial maps
\begin{eqnarray*}
  \Delta f:  && \Delta\mathcal{H}\longrightarrow \Delta\mathcal{H}',\label{eq-3bbd}\\
    \delta f:  && \delta\mathcal{H}\longrightarrow \delta\mathcal{H}'  \label{eq-3.bbh}
\end{eqnarray*}
 and  two  morphisms  of  independence  hypergraphs
 \begin{eqnarray*}
   \bar\Delta f:  &&\bar\Delta\mathcal{H}\longrightarrow\bar \Delta\mathcal{H}',\\
    \bar \delta f:  &&\bar \delta\mathcal{H}\longrightarrow \bar\delta\mathcal{H}'.
 \end{eqnarray*}
We   have  the following   naturalities:

(1).   
The  canonical  inclusions   $i_\Delta$  and   $i_\delta$    are   natural.  That  is,  for any  morphism  $f:  \mathcal{H}\longrightarrow  \mathcal{H}'$  we  have the commutative diagrams
\begin{eqnarray*}
\xymatrix{
\mathcal{H}\ar[d]_{i_\Delta}\ar[r]^{f}  &\mathcal{H}'\ar[d]^{i'_{\Delta}}\\
\Delta\mathcal{H}\ar[r]^{\Delta  f}  &\Delta\mathcal{H}',
}
~~~~~~~~~~~~
\xymatrix{
\mathcal{H}\ar[d]_{i_\delta}\ar[r]^{f}  &\mathcal{H}'\ar[d]^{i'_{\delta}}\\
\delta\mathcal{H}\ar[r]^{\delta  f}  &\delta\mathcal{H}'.
}
\end{eqnarray*}

(2).  
The  join   of  hypergraphs  is   natural.    That  is,  for
  any  two  morphisms  $f:  \mathcal{H}_1\longrightarrow \mathcal{H}_2$   and   $f':  \mathcal{H}'_1\longrightarrow \mathcal{H}'_2$,  we  have a  canonical  induced morphism
  $f  *  f':  \mathcal{H}_1  *\mathcal{H}'_1\longrightarrow  \mathcal{H}_2  *\mathcal{H}'_2$  such that  $(f  *  f')\mid_{\mathcal{H}_1}=f$, $(f  *  f')\mid_{\mathcal{H}'_1}=f'$  and  for any  $ \sigma\in \mathcal{H}_1$  and  any  $\sigma'\in \mathcal{H}'_1$,    the  join $\sigma*\sigma'$  in  $\mathcal{H}_1*\mathcal{H}'_1$   is sent  to  the  hyperedge
  $f(\sigma)*  f'(\sigma')$ in    $\mathcal{H}_2*\mathcal{H}'_2$.   The  followings are  direct
  \begin{eqnarray*}
&  \Delta(f*f')=\Delta  f  * \Delta  f',~~~~~~
\delta(f*f')=\delta  f  * \delta  f',\nonumber\\
& \bar\Delta(f*f')=\bar\Delta  f  * \bar\Delta  f',~~~~~~
 \bar\delta(f*f')=\bar\delta  f  * \bar\delta  f',  
\end{eqnarray*}
where     $\bar\Delta$ and $\bar\delta$   on the left-hand sides are with respect to $V\sqcup  V'$  while       $\bar\Delta$ and $\bar\delta$   on the right-hand sides are with respect to  $V$  and  $V'$.

(3).  
 The  box  product  of  hypergraphs  is   natural.  That  is,   for
  any  two  morphisms  $f:  \mathcal{H}_1\longrightarrow \mathcal{H}_2$   and   $f':  \mathcal{H}'_1\longrightarrow \mathcal{H}'_2$,  we  have a  canonical  induced morphism
  $f  \square  f':  \mathcal{H}_1  \square \mathcal{H}'_1\longrightarrow  \mathcal{H}_2  \square \mathcal{H}'_2$
   sending  $\sigma\square \sigma'$  to  $f(\sigma)\square  f'(\sigma')$.   % In  particular,

 \section{Random  hypergraphs   and   random  simplicial  complexes}\label{s3333333a}

\subsection{General  random  hypergraphs  and   random  simplicial  complexes}

Let  ${\bf H}(V)$  be the category whose objects are  hypergraphs on $V$   and whose morphisms are  morphisms of hypergraphs.  Let  ${\bf K}(V)$  be the category whose objects are    simplicial complexes on $V$   and whose morphisms are  simplicial maps.  Let  ${\bf  L}(V)$  be the category whose objects are  independence  hypergraphs on $V$   and whose morphisms are  morphisms of   independence  hypergraphs.
 Both  ${\bf K}(V)$   and  ${\bf  L}(V)$   are   full  subcategories   of  ${\bf  H}(V)$.

A  {\it  random  hypergraph}  (resp. {\it  random  simplicial  complex} and {\it  random  independence  hypergraph})  on  $V$   is   a  probability  function  on   ${\bf Obj}({\bf H}(V))$  (resp.  ${\bf Obj}({\bf  K}(V))$  and
 ${\bf Obj}({\bf  L}(V))$).
  Let  $D({\bf H}(V))$  (resp.  $D ({\bf K}(V)) $  and  $D ({\bf L}(V)) $)  be  the  functional space  of  all the  probability  functions
on  ${\bf Obj}({\bf H}(V))$  (resp.  ${\bf Obj}({\bf  K}(V))$  and
 ${\bf Obj}({\bf  L}(V))$).
 For  any  map
 \begin{eqnarray*}
 f: &&  {\bf H}(V)\longrightarrow  {\bf H}(V)
 \end{eqnarray*}
  there  is an induced map
  \begin{eqnarray*}
  Df: &&  D({\bf H}(V))\longrightarrow  D({\bf H}(V))
  \end{eqnarray*}
  given by
    \begin{eqnarray*}
  (Df)(\varphi)(\mathcal{H})=\sum_{f(\mathcal{H}')=\mathcal{H}} \varphi(\mathcal{H}'),
  \end{eqnarray*}
  where  $\varphi\in D({\bf H}(V))$  and  $\mathcal{H},\mathcal{H}'\in {\bf  Obj}({\bf H}(V))$.
  Let $f$  be $\Delta,\delta,\bar\Delta$ and $\bar\delta$ respectively.
  We  have the  induced maps
 \begin{eqnarray*}
  D\Delta,  D\delta:  &&D({\bf H}(V))\longrightarrow  D({\bf K}(V)),\\
      D\bar\Delta,  D\bar\delta:  &&D({\bf H}(V))\longrightarrow  D({\bf L}(V)).
 \end{eqnarray*}

Moreover,   
 for any map
 $\mu: {\bf Obj}({\bf H}(V))\times{\bf Obj}({\bf H}(V))\longrightarrow  {\bf Obj}({\bf H}(V))$   
 there is  an  induced map
  \begin{eqnarray*}
  D\mu:&&  D({\bf H}(V)) \times D({\bf H}(V)) \longrightarrow  D({\bf H}(V))
  \end{eqnarray*}
given  by
\begin{eqnarray*}
D\mu(\varphi',\varphi'')(\mathcal{H})=
\sum_{\mu(\mathcal{H}',\mathcal{H}'')=\mathcal{H}}
\varphi'(\mathcal{H}')\varphi''(\mathcal{H}''),
\end{eqnarray*}
where $\varphi',\varphi''\in  D({\bf H}(V))$  and   $\mathcal{H},  \mathcal{H}',\mathcal{H}''\in {\rm Obj}({\bf H}(V))$.
Let $\mu(\mathcal{H}',\mathcal{H}'')$  be  $\mathcal{H}'\cap\mathcal{H}''$ and
$\mathcal{H}'\cup\mathcal{H}''$
  respectively.
We  obtain  the  induced maps
 \begin{eqnarray*}
 D\cap,  D\cup:&&  D({\bf H}(V)) \times D({\bf H}(V)) \longrightarrow  D({\bf H}(V)).
 \end{eqnarray*}
The restrictions of  $D\cap$  and  $D\cup$     give the   maps
  \begin{eqnarray*}
 D\cap,  D\cup: && D({\bf K}(V)) \times D({\bf K}(V)) \longrightarrow  D({\bf K}(V)),\\
 D\cap,  D\cup:  && D({\bf L}(V)) \times D({\bf L}(V)) \longrightarrow  D({\bf L}(V)).
 \end{eqnarray*}

    Furthermore,  let $V'$  and  $V''$  be  two   finite sets.  Then
    the  box  product
  \begin{eqnarray*}
\square:  &&  {\bf Obj}({\bf H}(V'))\times  {\bf Obj}({\bf H}(V''))\longrightarrow  {\bf Obj}({\bf H}(V'\times V''))
 \end{eqnarray*}
 sending  $(\mathcal{H}',\mathcal{H}'')$ to  $\mathcal{H}'\square \mathcal{H}''$  induces a map
  \begin{eqnarray*}
 D\square:  &&  D({\bf H}(V'))\times  D({\bf H}(V''))\longrightarrow  D({\bf H}(V'\times  V''))
 \end{eqnarray*}
 given  by
  \begin{eqnarray}\label{eq-88902}
 (D\square)(\varphi',\varphi'')(\mathcal{H})= \sum_{\mathcal{H}'\square\mathcal{H}''=\mathcal{H}}  \varphi'(\mathcal{H}')\varphi''(\mathcal{H}'').
 \end{eqnarray}
 Here in   (\ref{eq-88902}),  we  take  $\varphi'\in   D({\bf H}(V'))$,   $\varphi''\in   D({\bf H}(V''))$,  $\mathcal{H}'\in {\bf Obj}({\bf H}(V'))$,   $\mathcal{H}''\in {\bf Obj}({\bf H}(V''))$   and  $\mathcal{H}\in {\bf Obj}({\bf H}(V'\times   V''))$.   
Suppose  in  addition $V'\cap  V''=\emptyset$.    Then
     the join
 \begin{eqnarray*}
\ast:  &&  {\bf Obj}({\bf H}(V'))\times  {\bf Obj}({\bf H}(V''))\longrightarrow  {\bf Obj}({\bf H}(V'\sqcup V''))
 \end{eqnarray*}
 sending  $(\mathcal{H}',\mathcal{H}'')$ to  $\mathcal{H}'*\mathcal{H}''$  induces a map
 \begin{eqnarray*}
 D\ast:  &&  D({\bf H}(V'))\times  D({\bf H}(V''))\longrightarrow  D({\bf H}(V'\sqcup  V''))
 \end{eqnarray*}
 given  by
 \begin{eqnarray}\label{eq-88901}
 (D\ast)(\varphi',\varphi'')(\mathcal{H})= \sum_{\mathcal{H}'*\mathcal{H}''=\mathcal{H}}  \varphi'(\mathcal{H}')\varphi''(\mathcal{H}'').   
 \end{eqnarray}
  Here in  (\ref{eq-88901}),  we  take  $\varphi'\in   D({\bf H}(V'))$,   $\varphi''\in   D({\bf H}(V''))$,  $\mathcal{H}'\in {\bf Obj}({\bf H}(V'))$,   $\mathcal{H}''\in {\bf Obj}({\bf H}(V''))$   and  $\mathcal{H}\in {\bf Obj}({\bf H}(V'\sqcup  V''))$.

\begin{lemma}
\label{le-3.aaa8}
\begin{enumerate}[(1).]
\item
 $(D\cup)(D\Delta,D\Delta)=(D\Delta)(D\cup)$   and     $(D\cup)(D\bar\Delta,D\bar\Delta)=(D\bar\Delta)(D\cup)$,   
\item
 $(D\cap)(D\delta,D\delta)=(D\delta)(D\cap)$   and      $(D\cap)(D\bar\delta,D\bar\delta)=(D\bar\delta)(D\cap)$;
\item
$(D\ast) (D\Delta,D\Delta)=  (D\Delta)(D\ast)$,  
$(D\ast) (D\delta,D\delta)=  (D\delta)(D\ast)$,  
$(D\ast) (D\bar\Delta,D\bar\Delta)=  (D\bar\Delta)(D\ast)$  and   $(D\ast) (D\bar\delta,D\bar\delta)=  (D\bar\delta)(D\ast)$.
\end{enumerate}
\end{lemma}

\begin{proof}
For any  $\varphi',\varphi''\in   D({\bf H}(V))$  and  any  $\mathcal{K}\in   {\bf K}(V) $,
 with the  help  of  Subsection~\ref{sect2.1111111}~(ii),
\begin{eqnarray*}
(D\cup)(D\Delta,D\Delta) (\varphi',\varphi'')(\mathcal{K})&=&\sum_{ \mathcal{K}'\cup\mathcal{K}''=\mathcal{K}}
\Big((D\Delta) (\varphi')(\mathcal{K}')\Big) \Big((D\Delta) (\varphi'')(\mathcal{K}'')\Big)\\
&=&\sum_{ \mathcal{K}'\cup\mathcal{K}''=\mathcal{K}}  \Big(\sum_{\Delta\mathcal{H}'=\mathcal{K}'} \varphi'(\mathcal{H}') \Big)  \Big(\sum_{\Delta\mathcal{H}''=\mathcal{K}''} \varphi''(\mathcal{H}'') \Big)\\
&=&\sum_{ \Delta\mathcal{H}'\cup\Delta\mathcal{H}''=\mathcal{K}}   \varphi'(\mathcal{H}') \varphi''(\mathcal{H}'')
\\
 &=&\sum_{\Delta(\mathcal{H}'\cup\mathcal{H}'')=\mathcal{K}}\varphi'(\mathcal{H}') \varphi''(\mathcal{H}'') \\
&=&\sum_{\Delta\mathcal{H}=\mathcal{K}}  \sum_{\mathcal{H}'\cup\mathcal{H}''=\mathcal{H}}  \varphi'(\mathcal{H}') \varphi''(\mathcal{H}'') \\
&=&(D\Delta)(D\cup) (\varphi',\varphi'')(\mathcal{K}).
\end{eqnarray*}
We  obtain $(D\cup)(D\Delta,D\Delta)=(D\Delta)(D\cup)$.
  Other identities can be proved analogously.  The second  identity  in (1)   is  proved with the  help  of  Subsection~\ref{sect2.1111111}~(iv).  The two identities in (2) are respectively proved   with the  help  of  Subsection~\ref{sect2.1111111}~(iii)  and  (v).  The  identities in (3)  are  proved with the help  of  Subsection~\ref{sect2.1111111}~(i)'  -  (iv)'.
\end{proof}

\subsection{Random hypergraphs   and  random  simplicial  complexes   of   Erd\"os-R\'enyi  type}\label{sss3.22222}

 Let   $p: \Delta[V]\longrightarrow  [0,1]$.
 Let $\mathcal{K}$  be  a  simplicial  complex on  $V$.
 An  {\it   external face}  of  $\mathcal{K}$
is  a  hyperedge  $\sigma\in\Delta[V]$  such that  $\sigma\notin \mathcal{K}$  and $\tau  \in  \mathcal{K}$  for any  nonempty
  proper subset  $\tau\subsetneq \sigma$.
  Let  $E(\mathcal{K})$  be the collection of  all the   external faces of $\mathcal{K}$.
 Let  $\mathcal{L}$  be an  independence hypergraph on  $V$.  A   {\it  co-external face}  of  $\mathcal{L}$
is  a  hyperedge  $\sigma\in\Delta[V]$  such that  $\sigma\notin \mathcal{L}$  and $\tau  \in  \mathcal{L}$  for any
  proper superset  $\tau\supsetneq \sigma$,  where $\tau\in\Delta[V]$.
  Let  $\bar E(\mathcal{L})$  be the collection of  all the  co-external faces of $\mathcal{L}$.

(1).  Consider  the   Erd\"os-R\'enyi-type  model
${\rm \bar  P}_p$  of  random hypergraphs  given by
\begin{eqnarray*}
\bar{\rm P}_p(\mathcal{H})=\prod_{\sigma\in \mathcal{H}}  p(\sigma) \prod_{\sigma\notin \mathcal{H}} \big(1-p(\sigma)\big)
\end{eqnarray*}
for any  $\mathcal{H}\in  {\bf Obj}({\bf H}(V))$.
We  choose  each  element $\sigma\in \Delta[V]$  to be  a  hyperedge  of $\mathcal{H}$  independently
at  random   with  probability  $p(\sigma)$.  This   randomly generated
hypergraph $\mathcal{H}$  satisfies the probability distribution  $\bar{\rm P}_p$,
 written   $\mathcal{H}\sim \bar{\rm P}_p$.

(2).  
Consider  the   Erd\"os-R\'enyi-type  model
${\rm  P}_p$  of  random  simplicial  complexes  given by
  \begin{eqnarray*}
  {\rm P}_{p}(\mathcal{K})=\prod_{\sigma\in\mathcal{K}}  p(\sigma)\prod_{\sigma\in  E(\mathcal{K})}\big(1-p(\sigma)\big),
  \end{eqnarray*}
  where   $ \mathcal{K}\in  {\rm Obj}({\bf K}(V))$.   We  generate  the  $0$-skeleton   of   $\mathcal{K}$   by choosing each $0$-hyperedge $\{v\}\in\Delta[V]$    independently  at random   with
  probability  $p(\{v\})$.    For any nonnegative  integer $k$,   once the  $k$-skeleton   of   $\mathcal{K}$   is    randomly generated,  we  generate  the  $(k+1)$-skeleton   of   $\mathcal{K}$  by choosing each  external  $(k+1)$-face  $\sigma$  of the  $k$-skeleton   of   $\mathcal{K}$   independently at random with  probability  $p(\sigma)$.  By  an  induction  on  $k$,  the  final randomly generated  simplicial complex  $\mathcal{K}$
  satisfies  the probability distribution  ${\rm  P}_{p}$,  written  $\mathcal{K}\sim {\rm  P}_{p}$.

 (3).  
  Consider  the   Erd\"os-R\'enyi-type  model
${\rm  Q}_p$  of  random  independence  hypergraphs  given by
  \begin{eqnarray*}
{\rm Q}_{p}(\mathcal{L})=\prod_{\sigma\in\mathcal{L}}  p(\sigma)\prod_{\sigma\in  \bar  E(\mathcal{L})}\big(1-p(\sigma)\big),
  \end{eqnarray*}
  where   $ \mathcal{L}\in  {\rm Obj}({\bf  L}(V))$.   
  We   choose  the  $(|V|-1)$-hyperedge   $V\in\Delta[V]$   at  random  with  probability  $p(V)$  
  to  be  a  hyperedge  of  $\mathcal{L}$.    
   Once all  the  $i$-hyperedges,  $i=k+1, k+2,\ldots, |V|-1$,    of   $\mathcal{L}$   are
  randomly  generated,   we  generate  the   $k$-hyperedges   of   $\mathcal{L}$   by choosing each  co-external face      $\sigma$  of the  collection of  all  the    $(k+1)$-hyperedges   of   $\mathcal{L}$     independently at random with  probability  $p(\sigma)$.   By   an  induction  on  $k$  in reverse,   The  final randomly generated  independence  hypergraph  $\mathcal{L}$
   satisfies  the probability distribution  ${\rm Q}_{p}$,  written  $\mathcal{L}\sim {\rm Q}_{p}$.

%By  (2)  and  (3),    for   any  $ \mathcal{K}\in  {\rm Obj}({\bf K}(V))$,  we  have
%\begin{eqnarray*}
%  {\rm P}_{p}(\mathcal{K}){\rm Q}_{p}(\gamma\mathcal{K})=\prod_{\sigma\in\Delta[V]}  p(\sigma)
%  \prod_{\sigma\in  E(\mathcal{K})  \cup \bar   E(\gamma\mathcal{K})}  \big(1-p(\sigma)\big).
%\end{eqnarray*}
%Equivalently,   for  any  $ \mathcal{L}\in  {\rm Obj}({\bf L}(V))$,  we  have
%\begin{eqnarray*}
%  {\rm  Q}_{p}(\mathcal{L}){\rm  P}_{p}(\gamma\mathcal{L})=\prod_{\sigma\in\Delta[V]}  p(\sigma)
%  \prod_{\sigma\in   \bar   E(\mathcal{L}) \cup   E(\gamma  \mathcal{L}) }  \big(1-p(\sigma)\big).
%\end{eqnarray*}

\begin{lemma}{\rm  (cf.  \cite[Lemma~4.4]{jktr2})}.
\label{cplmt}
$(D\gamma)(\bar{\rm   P}_p)= \bar{\rm   P}_{1-p} $.   
\label{le-3.22876}
\end{lemma}

\begin{lemma}{\rm  (cf. \cite[Lemma~4.5]{jktr2})}.
\label{le-3.1aa}
Let $p',p'': \Delta[V]\longrightarrow [0,1]$.  Write $(p'\cap p'')(\sigma)= p'(\sigma) p''(\sigma)$  and   $(p'\cup p'')(\sigma)= 1-(1-p'(\sigma))(1- p''(\sigma))$  for any $\sigma\in\Delta[V]$.
 Then  $(D\cap)(\bar {\rm P}_{p'},  \bar  {\rm  P}_{p''}) =  \bar  {\rm  P}_{p'\cap p''}$  and   $(D\cup)(\bar {\rm  P}_{p'},  \bar  {\rm  P}_{p''}) =  \bar  {\rm  P}_{p'\cup p''}$.  In  other words,   if  $\mathcal{H}'\sim \bar {\rm P}_{p'}$  and $\mathcal{H}''\sim \bar {\rm P}_{p''}$,       then
$\mathcal{H}'\cap \mathcal{H}''\sim  \bar {\rm P}_{p'\cap p''}$   and  $\mathcal{H}'\cup \mathcal{H}''\sim  \bar {\rm  P}_{p'\cup p''}$.
\end{lemma}

\begin{lemma}\label{pr-88.a8}
  Let  $V'$  and  $V''$  be  two disjoint vertex sets.  Let  $p': \Delta[V']\longrightarrow [0,1]$  and $p'': \Delta[V'']\longrightarrow [0,1]$.
Write $(p'*p'')(\sigma)=p'(\sigma\cap V') p''(\sigma\cap V'')$  for any $\sigma\in \Delta[V'\sqcup V'']$.
 Then
 $(D\ast)(\bar {\rm P}_{p'},\bar {\rm P}_{p''})=  \bar {\rm P}_{p'*p''}$.
  In  other words,   if  $\mathcal{H}'\sim \bar {\rm P}_{p'}$  and $\mathcal{H}''\sim \bar {\rm P}_{p''}$,       then
 $\mathcal{H}'*\mathcal{H}''\sim  \bar {\rm P}_{p'*p''}$.
  \end{lemma}

\begin{proof}
Consider the  following  two   independent   trials:    (1).  generate  $\mathcal{H}'$,  (2).  generate $\mathcal{H}''$.
Then   for  any  $\sigma\in \Delta[V'\sqcup  V'']$,  both of the followings hold:
\begin{enumerate}[(1).]
\item
$\sigma\in  \mathcal{H}'*\mathcal{H}''$  iff  both $(\sigma\cap  V')\in  \mathcal{H}'$  in trial  (1)  and
 $(\sigma\cap  V'')\in  \mathcal{H}''$  in  trial  (2).   Thus  the event $\sigma\in  \mathcal{H}'*\mathcal{H}''$  is the product  event  of  the two  independent  events  $(\sigma\cap  V')\in  \mathcal{H}'$  and  $(\sigma\cap  V'')\in  \mathcal{H}''$.  This event  has  the  probability $p'(\sigma\cap  V')p''(\sigma\cap  V'')$;

 \item
 $\sigma\notin  \mathcal{H}'*\mathcal{H}''$  iff  either  $(\sigma\cap  V')\notin  \mathcal{H}'$  in trial  (1)  or
 $(\sigma\cap  V'')\notin  \mathcal{H}''$  in  trial  (2).  Thus  the event  $\sigma\notin  \mathcal{H}'*\mathcal{H}''$  has  the  probability $1-p'(\sigma\cap  V')p''(\sigma\cap  V'')$.
 \end{enumerate}
Therefore,  $\mathcal{H}'*\mathcal{H}''$  is  the randomly generated hypergraph  on  $V'\sqcup V''$   satisfying  the probability distribution  $ \bar {\rm P}_{p'*p''}$.
\end{proof}

 \begin{theorem}
 \label{le-3.3aa}
 For  any   $\mathcal{H}\in {\rm Obj}({\bf{H }}(V))$,  any   $\mathcal{K}\in {\rm Obj}({\bf{K }}(V))$  and   any   $\mathcal{L}\in {\rm Obj}({\bf{L }}(V))$,
  \begin{enumerate}[(1).]
  \item
  $ \big((D\Delta)(\bar{\rm P}_{p})\big)(\mathcal{K})=
    {\rm  Q}_{1-p}(\gamma\mathcal{K})$,
     \item
     $ \big((D\bar\Delta)(\bar{\rm P}_{p})\big)(\mathcal{L})=
   {\rm  P}_{1-p}(\gamma\mathcal{L})$,
  \item
  $ \big((D\delta)(\bar{\rm P}_{p})\big)(\mathcal{K})
         =  {\rm  P}_p(\mathcal{K})$,
\item
        $ \big((D\bar\delta)(\bar{\rm P}_{p})\big)(\mathcal{L})
        ={\rm Q}_{p}(\mathcal{L})$.
   \end{enumerate}
 \end{theorem}
 
We  will  prove  Theorem~\ref{le-3.3aa}  in the next  subsection.   The  following  corollary  is  a  restatement  of  
Theorem~\ref{le-3.3aa}~(1)  and  (2).

\begin{corollary}
\label{co-1y3058}
For  any   $\mathcal{H}\in {\rm Obj}({\bf{H }}(V))$,  any   $\mathcal{K}\in {\rm Obj}({\bf{K }}(V))$  and   any   $\mathcal{L}\in {\rm Obj}({\bf{L }}(V))$,
  \begin{enumerate}[(1).]
  \item
  $ \big((D\gamma)(D\Delta)(\bar{\rm P}_{p})\big)(\mathcal{L})=
    {\rm  Q}_{1-p}(\mathcal{L})$,
     \item
     $ \big((D\gamma)(D\bar\Delta)(\bar{\rm P}_{p})\big)(\mathcal{K})=
   {\rm  P}_{1-p}(\mathcal{K})$.
   \end{enumerate}
\end{corollary}

  The  following  two  corollaries   follow   from  
Theorem~\ref{le-3.3aa}.

\begin{corollary}\label{co-nwtllip}
Let $p',p'': \Delta[V]\longrightarrow [0,1]$.    Then
\begin{enumerate}[(1).]
\item
$(D\cap ) ({\rm  P}_{p'},{\rm  P}_{p''})={\rm  P}_{p'\cap  p''}$,
\item
$(D\cap ) ({\rm  Q}_{p'},{\rm  Q}_{p''})={\rm  Q}_{p'\cap  p''}$.
\end{enumerate}
\label{pr-3.8a1}
\end{corollary}
\begin{proof}
By  Lemma~\ref{le-3.aaa8},  Lemma~\ref{le-3.1aa}  and   Theorem~\ref{le-3.3aa},
\begin{eqnarray*}
(D\cap )  ({\rm  P}_{p'},  {\rm  P}_{p''}) &=&  (D\cap  )(D\delta,D\delta)(\bar{\rm  P}_{p'},  \bar{\rm  P}_{p''}) \\
&=&(D\delta)(D\cap) (\bar{\rm  P}_{p'},  \bar{\rm  P}_{p''}) \\
&=&(D\delta)(\bar{\rm  P}_{p'\cap  p''})\\
&=&{\rm  P}_{p'\cap  p''}
\end{eqnarray*}
and
\begin{eqnarray*}
(D\cap )  ({\rm  Q}_{p'},  {\rm  Q}_{p''}) &=&  (D\cap  )(D\bar\delta,D\bar\delta)(\bar{\rm  P}_{p'},  \bar{\rm  P}_{p''}) \\
&=&(D\bar\delta)(D\cap) (\bar{\rm  P}_{p'},  \bar{\rm  P}_{p''}) \\
&=&(D\bar\delta)(\bar{\rm  P}_{p'\cap  p''})\\
&=&{\rm  Q}_{p'\cap  p''}.
\end{eqnarray*}
\end{proof}

\begin{corollary}\label{co-1l3n5rdvg}
  Let  $V'$  and  $V''$  be  two disjoint vertex sets.  Let  $p': \Delta[V']\longrightarrow [0,1]$  and $p'': \Delta[V'']\longrightarrow [0,1]$.   Then
  \begin{enumerate}[(1).]
\item
  $(D\ast ) ({\rm  P}_{p'},{\rm  P}_{p''})={\rm  P}_{p'*  p''}   $,
\item
   $(D\ast ) ({\rm  Q}_{p'},{\rm  Q}_{p''})={\rm  Q}_{p'*  p''}$.
  \end{enumerate}
\label{pr-3.8a2}
\end{corollary}
\begin{proof}
By  Lemma~\ref{le-3.aaa8},  Lemma~\ref{pr-88.a8}  and   Theorem~\ref{le-3.3aa},
\begin{eqnarray*}
(D*)({\rm  P}_{p'},  {\rm  P}_{p''})&=&   (D*  )(D\delta,D\delta)(\bar{\rm  P}_{p'},  \bar{\rm  P}_{p''}) \\
&=&(D\delta)(D*) (\bar{\rm  P}_{p'},  \bar{\rm  P}_{p''}) \\
&=&(D\delta)(\bar{\rm  P}_{p'*  p''})\\
&=&{\rm  P}_{p'*  p''}
\end{eqnarray*}
and
\begin{eqnarray*}
(D*)  ({\rm  Q}_{p'},  {\rm  Q}_{p''}) &=&  (D*  )(D\bar\delta,D\bar\delta)(\bar{\rm  P}_{p'},  \bar{\rm  P}_{p''}) \\
&=&(D\bar\delta)(D*) (\bar{\rm  P}_{p'},  \bar{\rm  P}_{p''}) \\
&=&(D\bar\delta)(\bar{\rm  P}_{p'* p''})\\
&=&{\rm  Q}_{p'*  p''}.
\end{eqnarray*}
\end{proof}

\subsection{Proof of  Theorem~\ref{le-3.3aa}}

Consider    a   map  $p:  \Delta[V]\longrightarrow  [0,1]$.   %  and the   random   hypergraph    satisfying   the probability distribution  $\bar{\rm P}_p$.    
 For  any  hypergraph   $\mathcal{H}$  on  $V$  
 and  any  hyperedge     $\sigma\in \mathcal{H}$,   we  call  $\sigma$  a   {\it  maximal  hyperedge}   of  $\mathcal{H}$  if     there  does  not  exist  any  $\tau\in \mathcal{H}$  such  that 
$\sigma\subsetneq  \tau$   and  call   $\sigma$  a   {\it  minimal   hyperedge}   of  $\mathcal{H}$  if     there  does  not  exist  any  $\tau\in \mathcal{H}$  such  that 
$\sigma\supsetneq  \tau$.    Let  $\max(\mathcal{H})$  be  the  collection  of  all the  maximal  hyperedges  of  
 $\mathcal{H}$  and   let  $\min(\mathcal{H})$  be  the  collection  of  all the  minimal  hyperedges  of  
$\mathcal{H}$.

\begin{lemma}\label{le-er1}
{\rm  (cf.   \cite[Theorem~1.5~(2)]{jktr2})}.   
     For any  simplicial   complex $\mathcal{K}$  on  $V$,  
 \begin{eqnarray}
    \big((D\Delta)(\bar{\rm P}_{p})\big)(\mathcal{K})=
     \prod_{\tau\in\max(\mathcal{K})} p(\tau)   \prod_{\tau\in\Delta[V]\atop \tau\notin \mathcal{K}} \big(1-p(\tau) \big).  
      \label{eq-aoq2}
     \end{eqnarray}
\end{lemma}

The  proof of  Lemma~\ref{le-er1}  is  in  \cite[Lemma~4.2]{jktr2}.   
  \begin{lemma}\label{le-er2}
        For any  independence  hypergraph   $\mathcal{L}$  on  $V$,  
    \begin{eqnarray}
     \big((D\bar\Delta)(\bar{\rm P}_{p})\big)(\mathcal{L})=
     \prod_{\tau\in\min(\mathcal{L})} p(\tau)   \prod_{\tau\in\Delta[V]\atop  \tau\notin \mathcal{L}} \big(1-p(\tau) \big).  
       \label{eq-aoq3}
  \end{eqnarray}
  \end{lemma}
  \begin{proof}
  Let  $\mathcal{L}$  be an  independence  hypergraph on   $V$.
   Let $S=\{\sigma_1,\sigma_2,\ldots, \sigma_s\}$  be any set (allowed to be the emptyset) of distinct hyperedges in  $\mathcal{L}$ such that for each $\sigma_i$,  $i=1,2,\ldots,s$,  there exists  $\tau\in\min(\mathcal{L})$  such  that $\sigma_i\supsetneq \tau$.  Suppose  $S$  runs  over  all  such sets
   of  hyperedges  in $\mathcal{L}$.  Then  $\mathcal{H}=\min(\mathcal{L})\sqcup S$  runs over  all the   hypergraphs
    on  $V$  such that  $\bar\Delta\mathcal{H}=\mathcal{L}$.
    Consequently,  
        \begin{eqnarray}
    \big((D\bar\Delta)(\bar{\rm P}_{p})\big)(\mathcal{L})&=& \sum_{\bar\Delta \mathcal{H}=\mathcal{L}}~\bar{\rm P}_{p}(\mathcal{H})
    \nonumber\\
    &=& \sum_{\bar\Delta \mathcal{H}=\mathcal{L}}~\prod_{\sigma\in\mathcal{H}}  p(\sigma)\prod_{\sigma\notin   \mathcal{H}}\big(1-p(\sigma)\big)\nonumber\\
    &=& \sum_{\mathcal{H}=\min(\mathcal{L})\sqcup S}~\prod_{\sigma\in\mathcal{H}}  p(\sigma)\prod_{\sigma\in\Delta[V]  \atop\sigma\notin   \mathcal{H}}\big(1-p(\sigma)\big)
    \nonumber\\
    &=& \prod_{\sigma\in\min(\mathcal{L})}  p(\sigma)
     \prod_{\sigma\in\Delta[V]  \atop\sigma\notin   \mathcal{L}}\big(1-p(\sigma)\big)
    \nonumber\\
&&      \Big(\sum_{S\subseteq  \mathcal{L}\setminus\min(\mathcal{L})}~
    \prod_{\sigma\in  S} p(\sigma)\prod_{\sigma\in \mathcal{L}\setminus \min(\mathcal{L})\atop \sigma\notin S}
    \big(1-p(\sigma)\big)\Big)   \nonumber\\
     &=& \prod_{\sigma\in\min(\mathcal{L})}  p(\sigma)
     \prod_{\sigma\in\Delta[V]  \atop\sigma\notin   \mathcal{L}}\big(1-p(\sigma)\big)\nonumber\\
     &&   \prod_{\sigma\in \mathcal{L}\setminus\min(\mathcal{L})} \Big(p(\sigma)+\big(1-p(\sigma)\big)\Big) \nonumber\\
    &=& \prod_{\tau\in\min(\mathcal{L})} p(\tau)   \prod_{\tau\in\Delta[V] \atop \tau\notin \mathcal{L}} \big(1-p(\tau) \big).
    \label{eq-lpgwq88}
    \end{eqnarray}
   We  obtain  (\ref{eq-aoq3}).
  \end{proof}

\begin{lemma}
\label{le-jdngt}
   For  any   simplicial  complex   $\mathcal{K}$  on  $V$,
\begin{eqnarray*}
 \big((D\delta)(\bar{\rm P}_{p})\big)(\mathcal{K})
         =  \prod_{\tau\in \min (\gamma\mathcal{K})} \big(1-p(\tau)\big)\prod_{\tau\in\mathcal{K}}  p(\tau).  
\end{eqnarray*}
\end{lemma}

\begin{proof}
By  a  direct  calculation,   
\begin{eqnarray*}
\big((D\delta)(\bar{\rm P}_{p})\big)(\mathcal{K})&=&  \big((D\gamma)\circ (D\bar\Delta)\circ (D\gamma)(\bar{\rm P}_{p})\big)(\mathcal{K})\\
&=&\sum_{\gamma \mathcal{L}=\mathcal{K}} (D\bar\Delta\circ D\gamma)(\bar{\rm P}_p)(\mathcal{L})\\
&=&\sum_{\gamma \mathcal{L}=\mathcal{K}} ~
\sum_{\bar\Delta \mathcal{H}=\mathcal{L}}( D\gamma)(\bar{\rm P}_p)(\mathcal{H})\\
&=&\sum_{\gamma \mathcal{L}=\mathcal{K}} ~
\sum_{\bar\Delta \mathcal{H}=\mathcal{L}}~
\sum_{\gamma\mathcal{H}'=\mathcal{H}} \bar{\rm P}_p(\mathcal{H}')\\
&=&\sum_{\bar \Delta \mathcal{H}=\gamma\mathcal{K}} \bar{\rm  P}_p (\gamma\mathcal{H})\\
&=& \prod_{\tau\in\min(\gamma\mathcal{K})} \big(1-p(\tau)\big)   \prod_{\tau\in \mathcal{K}}  p(\tau).
\end{eqnarray*}
The  last equality follows by a analogous calculation of (\ref{eq-lpgwq88}). 
\end{proof}

   \begin{lemma}\label{co-i3ro}
   For  any   independence  hypergraph  $\mathcal{L}$  on  $V$,
   \begin{eqnarray*}
    \big((D\bar\delta)(\bar{\rm P}_{p})\big)(\mathcal{L})
        =\prod_{\tau\in \max (\gamma\mathcal{L})} \big (1-p(\tau)\big)\prod_{\tau\in\mathcal{L}}  p(\tau).  
   \end{eqnarray*}
   \end{lemma}

\begin{proof}
By  a  direct  calculation,     
\begin{eqnarray*}
  \big((D\bar\delta)(\bar{\rm P}_{p})\big)(\mathcal{L})&=& \big((D\gamma)\circ (D\Delta)\circ (D\gamma)(\bar{\rm P}_{p})\big)(\mathcal{L})\\
&=&\sum_{\gamma \mathcal{K}=\mathcal{L}} (D\Delta\circ D\gamma)(\bar{\rm P}_p)(\mathcal{K})\\
&=&\sum_{\gamma \mathcal{K}=\mathcal{L}} ~
\sum_{\Delta \mathcal{H}=\mathcal{K}}( D\gamma)(\bar{\rm P}_p)(\mathcal{H})\\
&=&\sum_{\gamma \mathcal{K}=\mathcal{L}} ~
\sum_{\Delta \mathcal{H}=\mathcal{K}}~
\sum_{\gamma\mathcal{H}'=\mathcal{H}} \bar{\rm P}_p(\mathcal{H}')\\
&=&\sum_{ \Delta \mathcal{H}=\gamma\mathcal{L}} \bar{\rm  P}_p (\gamma\mathcal{H})\\
&=& \prod_{\tau\in\max(\gamma\mathcal{L})} \big(1-p(\tau)\big)   \prod_{\tau\in \mathcal{L}}  p(\tau).
\end{eqnarray*}
The  last equality follows by a analogous calculation of  \cite[Lemma~4.2]{jktr2}.
\end{proof}

 \begin{proof}[Proof of  Theorem~\ref{le-3.3aa}]
  For any  simplicial  complex  $\mathcal{K}$  on  $V$,  
  \begin{eqnarray*}
  E(\mathcal{K})=\min (\gamma\mathcal{K}).
  \end{eqnarray*}
 For  any  independence  hypergraph  $\mathcal{L}$  on  $V$,   
 \begin{eqnarray*}
 \bar  E(\mathcal{L})=\max (\gamma  \mathcal{L}).
 \end{eqnarray*}
 Therefore,  Theorem~\ref{le-3.3aa}  follows  from  Lemma~\ref{le-er1}   -   Lemma~\ref{co-i3ro}.   
 \end{proof}

\section{Generations       of  random hypergraphs  and   random  simplicial  complexes  by  the  map  algebra}\label{ss8ss88}

\subsection{A   graded    construction  of   the  map  algebra}

Let  ${\bf  H}$   be   the  category  such that  
each   object     is    a   pair   $(V,\mathcal{H})$,   where 
$V$  is  a  finite  set  and  $\mathcal{H}$  is a   hypergraph  on  $V$,  
  and
  each   morphism  from    $(V,\mathcal{H})$   to   $(V',\mathcal{H}')$   is   a   morphism  $f:  V\longrightarrow  V'$   of  hypergraphs    from   $\mathcal{H}$  to $\mathcal{H}'$.   
Let  $G^1$  be  the semi-group  generated  by  $\gamma$,  $\delta$  and  $\Delta$,    where the  unit  element  is  the  identity map and the  multiplication  is  the  composition  of  maps.      
 Each  element  of  $G^1$  is  a    map  from   ${\bf  Obj}({\bf  H})$   to  itself.     
 \begin{proposition}\label{th-ma25}
    Each  of         the  triples
%\begin{eqnarray*}
   $\{\gamma,\delta,\Delta\}$,   $\{\gamma,\bar\delta,\bar\Delta\}$,  $\{\gamma,\Delta,\bar\Delta\}$   and      $\{\gamma,\delta,\bar\delta\}$  
% \end{eqnarray*}
    is  a  set  of  multiplicative  generators  of  $G^1$.   
    \end{proposition}
    
    \begin{proof}
     By   Lemma~\ref{le-882890},  each of  the  triples   $\{\gamma,\Delta,\delta\}$, $\{\gamma,\bar\Delta,\bar\delta\}$,
$\{\gamma, \Delta, \bar\Delta\}$  and  $\{\gamma,\delta, \bar\delta\}$   could  multiplicatively  generate
$\gamma$,  $\Delta$,  $\delta$,  $\bar\Delta$  and $\bar\delta$.   Therefore,  each  of  the triple  could  multiplicatively
generate  $G^1$.   
    \end{proof}
         
    For  any  positive  integer  $k$,   Let  $G^k$   be  the  collection  of  all  the   words 
    \begin{eqnarray*}
     (\cdots (w_1\bullet  w_2)\bullet\cdots   \bullet  w_k),      
    \end{eqnarray*}
    where  $w_1,w_2,\ldots,  w_k\in  G$  and  $\bullet =  \cap,\cup,*$  or  $\square$,    with  the  binary  operation  $\bullet$   for  $k-1$  times and  $k-2$  brackets 
     giving  the  order  of  operations.   For  any  $w\in  G^k$,   
     we  call an  element  $(\cdots,  (V_1,\mathcal{H}_1),  (V_2,\mathcal{H}_2),\cdots )$   in   ${\bf  Obj}({\bf  H})^{\times k }$   {\it  $w$-admissible}  
     if   it  is  sent to  an  element  in ${\bf  Obj}({\bf  H})$  by  $w$,  i.e.    all  the  binary  operations 
     $w_1(\mathcal{H}_1)\bullet  w_2(\mathcal{H}_2)$  in   $w$   are  well-defined.   For  example,  
    any  element  in  ${\bf  Obj}({\bf  H})^{\times 2 }$   is   $\square$-admissible,  a   $*$-admissible  element  in 
    ${\bf  Obj}({\bf  H})^{\times 2 }$   is  of  the  form   $((V_1,\mathcal{H}_1),  (V_2,\mathcal{H}_2))$  where 
    $V_1\cap  V_2=\emptyset$,  
    and   a   $\cap$-admissible  element as  well as   a  $\cup$-admissible  element  in  ${\bf  Obj}({\bf  H})^{\times 2 }$
     is  of  the  form  $( (V,\mathcal{H}_1),  (V,\mathcal{H}_2))$.       
       Note  that 
     $w$   is   a   map  from   the  set  of all  $w$-admissible  elements,    which  is    a  subset  of  ${\bf  Obj}({\bf  H})^{\times k }$,  to  ${\bf  Obj}({\bf  H})$.           
          We  define   the  {\it  map  algebra}    to  be  the     union   
   \begin{eqnarray}\label{eq-4.jnuoh}
    G=  \bigcup_{k\geq 1}  G^k.
    \end{eqnarray}

  Let   $T$   be   a  binary  tree.     Let   $x_0$   be   the  root  of  $T$.  
  We  label   $x_0$   with  $\cap$,  $\cup$,  $*$  or  $\square$   if   $\deg  x_0  =2$   and  label  
  $x_0$  with  $\Delta$,  $\delta$  or   $\gamma$   if   $\deg  x_0=0,1$.      
For  any   vertex  $x$  of   $T$  with  $x\neq  x_0$,        
we  label  $x$  with  $\cap$,  $\cup$,  $*$  or  $\square$   if   $\deg  x   =3$   and  label  $x$  
with   $\Delta$,  $\delta$  or   $\gamma$   if   $\deg  x =1,2$.
 Let   $V_T$   be   the  vertex  set  of  $T$.   Consider   a   map  
 \begin{eqnarray*}
 \alpha:  &&  V_T \longrightarrow  \{\cap,\cup, *, \square,  \Delta,\delta,\gamma\}   
 \end{eqnarray*}
    sending   each  vertex  of   $T$  to  its  label  satisfying     the  above  labeling rule.  
    Let  $k(T)=1$   if   $T$   is  the  empty  binary  tree  or   $T$  has  a  single  vertex   $x_0$   and  let  $k(T)$       be   the  number  of  vertices  $x\neq  x_0$      such  that        $\deg  x  =1$     if   $T$  has  at  least  two  vertices.    
    Then  each   pair   $(T,\alpha)$    represents   an  element  in   $G^{k(T)}$.    
        For  each  nonnegative  integer $k$,  
       each  element  in  $G^k$  has  such  a  representation   $(T,\alpha)$,  while  two  pairs  $(T,\alpha)$  and  $(T',\alpha')$
        may  represent one  element  in  $G^k$.   The  pair  $(\emptyset,  \alpha)$,  where  $\emptyset$  is the  empty  binary  tree,  represents  ${\rm  id}\in  G^1$.   A  pair  $(T, \alpha)$,   where    $T$  is  a  binary  tree  such that  $\deg  x_0=1$  and  $\deg  x=1,  2$  for  any  $x\neq  x_0$,    represents  an  element  in  $G^1$.

       Let  ${\bf  DH}$   be   the  category  such that  
each    object    is  a    triple   $(V,\mathcal{H}, \varphi)$,  where  $V$  is   a  finite  set,   $\mathcal{H}$  is a hypergraph  on  $V$  and  $\varphi\in D({\bf H}(V))$  is a  probability distribution  on  ${\bf  Obj}({\bf H}(V))$,   and   each   morphism  from   $(V,\mathcal{H}, \varphi)$   to  $(V',\mathcal{H}', \varphi')$   
         is    a  morphism  $f:  V\longrightarrow  V'$  of  hypergraphs   from  $\mathcal{H}$  to   $\mathcal{H}'$  satisfying   
         $(Df)(\varphi)=\varphi'$.               
Let  $S$   be  the  subset  of         ${\bf  Obj}({\bf  DH})$  consisting  of  all the  elements   
         $(V,\mathcal{H}, \bar{\rm  P}_p)\in  {\bf  Obj}({\bf  DH})$,  
              where   $p:  \Delta[V]\longrightarrow [0,1]$.  
             For  any  $w_1,w_2\in   G^1$,  define  
         $D(w_1w_2)=(Dw_1)(Dw_2)$.      
         For  each  $w\in  G^1$,  define  
         \begin{eqnarray*}
         w(V,\mathcal{H}, \varphi)=  (w(V,\mathcal{H}),  (Dw) (\varphi)).  
         \end{eqnarray*}
         Then  $w$  is  a map   from   ${\bf  Obj}({\bf  DH})$  to  itself.   
         Let  $G^1(S)$   be  the  union   of  $w(S)$  for  all  $w\in   G_1$,   where  $w(S)$   is  the 
         image  of  $S$  under  $w$.  
         
         \begin{proposition}
         For  any  $p:  \Delta[V]\longrightarrow  [0,1]$,  
         ${\rm   P}_p$   and   and  ${\rm  Q}_p$  are   in  $G^1(S)$.  
         \end{proposition}
         
         \begin{proof}
         By  Theorem~\ref{th-mmmmm11111}~(4)   or  Theorem~\ref{le-3.3aa}~(3),    
         ${\rm   P}_p\in   G^1(S)$.       By  Theorem~\ref{th-mmmmm11111}~(5)   or  Theorem~\ref{le-3.3aa}~(4),    
         ${\rm   Q}_p\in   G^1(S)$.   
         \end{proof}

         Define   $D\square$   to  be   the  map   from      ${\bf  Obj}({\bf  DH})\times {\bf  Obj}({\bf  DH})$  to  ${\bf  Obj}({\bf  DH})$  by 
         \begin{eqnarray*}
         (D \square)((V,\mathcal{H},\varphi), (V',\mathcal{H}',\varphi'))=(V\times  V',  \mathcal{H}\square \mathcal{H}',  (D\square)(\varphi,\varphi')).      
            \end{eqnarray*}    
         Define       $D*$   to  be   the  map   from   the     collection    of    $*$-admissible  elements   in  ${\bf  Obj}({\bf  DH})\times {\bf  Obj}({\bf  DH})$  to  ${\bf  Obj}({\bf  DH})$  by 
         \begin{eqnarray*}
         (D *)((V,\mathcal{H},\varphi), (V',\mathcal{H}',\varphi'))=(V\sqcup  V',  \mathcal{H} *\mathcal{H}',  (D*)(\varphi,\varphi')),
                  \end{eqnarray*}
where  $V\cap  V'=\emptyset$.  
 Define     $D\cap$   and    $D\cup$  to  be    maps  from   the    collection    of    $\cap$-admissible  (equivalently,   $\cup$-admissible)   elements   in  ${\bf  Obj}({\bf  DH})\times {\bf  Obj}({\bf  DH})$  to  ${\bf  Obj}({\bf  DH})$  by 
         \begin{eqnarray*}
          (D \cap)((V,\mathcal{H},\varphi), (V,\mathcal{H}',\varphi'))&=&(V,  \mathcal{H}\cap \mathcal{H}',  (D\cap)(\varphi,\varphi')),\\
           (D \cup)((V,\mathcal{H},\varphi), (V,\mathcal{H}',\varphi'))&=&(V,  \mathcal{H}\cup \mathcal{H}',  (D\cup)(\varphi,\varphi')). 
         \end{eqnarray*}
        Define   $D(w_1\bullet  w_2)= (D\bullet)  (Dw_1, Dw_2)$  for  $\bullet =  \cap,\cup,*$  or  $\square$.  
         Then   each element  in  $G^{k}$       is   a  map  from    an   admissible  subset  of    ${\bf  Obj}({\bf   DH})^{\times  k}$  to  ${\bf  Obj}({\bf  DH})$.   
         Let 
         \begin{eqnarray*}
         G^k (S^{\times k}) &=& \{w((V_1,\mathcal{H}_1, \bar{\rm  P}_{*}),  \cdots,    (V_k,\mathcal{H}_k, \bar{\rm  P}_{*}))\mid 
         w\in  G^k {\rm~and~}\\
         &&((V_1,\mathcal{H}_1, \bar{\rm  P}_{*}),  \cdots,    (V_k,\mathcal{H}_k, \bar{\rm  P}_{*})) \in  S^{\times  k} ~{\rm~is~}w{\rm\text{-}admissible} \}. 
         \end{eqnarray*}
         Take      the     union  
         \begin{eqnarray}\label{eq-9823u592}
         G(S)= \bigcup_{k\geq  1} G^k (S^{\times k}).    
         \end{eqnarray}
                 Each  element   in    $G(S)$  is  of  the  form    $(V,\mathcal{H},\varphi)$,   where   $V$    is  a  finite   vertex  set,   
            $\mathcal{H}$  is  a  
         hypergraph   on $V$  and  
            $\varphi$   is  a   probability  function  on  ${\bf  Obj}({\bf H}(V))$.    The  probability  function   $\varphi$  is  given  by  the  action  of   certain     $D w$,  where  $w\in  G$,   on   multiple   probability   functions   of  the  form  $\bar  {\rm P}_*$,  ${\rm   P}_*$  and  ${\rm  Q}_*$.

    \subsection{Algorithms  for  the  generations  of   random hypergraphs  and  random  simplicial complexes}
    
    %We   give  an  algorithm   that  generates     random  hypergraphs  by  using  the  elements  in the map  algebra.      
    
    \begin{proof}[Algorithm generating  random hypergraphs]
       {\bf  Step~1}.       Choose  a  positive  integer  $k$.  
    Choose  $k$  finite  sets  $V_1$,  $V_2$,  $\ldots$,  $V_{k}$  as  the  vertex  sets  without   mutual  intersections.

    {\bf  Step~2}.  For each  $i=1,2,\ldots,  k$,   choose  a  positive  integer   $n_i$.  
      Use  the   Erd\"os-R\'enyi-type  model
${\rm \bar  P}_p$  to  give  $n_i$  randomly generated   hypergraphs  $\mathcal{H}_{1,i}$,  $\mathcal{H}_{2,i}$,  $\ldots$,  $\mathcal{H}_{n_i,i}$   on  $V_i$.

     {\bf  Step~3}.   For  each  $i=1,2,\ldots,  k$  and  each  $j=1,2,\ldots,  n_i$,  choose  an  element  $w_{j,i}\in  G^1$.   Apply  $w_{j,i}$  to  
      $\mathcal{H}_{j,i}$  and    give  a  randomly  generated  hypergraph  $w_{j,i}(\mathcal{H}_{j,i})$  on   $V_i$.

   {\bf  Step~4}.     For  each  $i=1,2,\ldots,  k$,   
   {\bf  (4.a)}.  Choose  two  randomly  generated hypergraphs  from  $w_{j,i}(\mathcal{H}_{j,i})$,   $j=1$, $2$, $\ldots$,  $n_i$,  and  apply  $\cap$  or  $\cup$  to these  two  randomly  generated  hypergraphs.  
   {\bf  (4.b)}.  Choose  a randomly  generated hypergraph  from the remaining  $n_i-2$  randomly generated  hypergraphs    $w_{j,i}(\mathcal{H}_{j,i})$,   $j=1$, $2$, $\ldots$,  $n_i$,  and  apply  $\cap$  or $\cup$  to  this    chosen   randomly  generated hypergraph  and  the  randomly  generated  hypergraph  given  in  (4.a).       
   {\bf  (4.c)}.  Repeat   (4.b)  for  $n_i-2$  times.   Denote  this  final  randomly  generated  hypergraph  on  $V_i$      as  $\mathcal{H}_i$.

      {\bf  Step~5}.    {\bf  (5.a)}.  Choose  two  randomly  generated hypergraphs  from  $\mathcal{H}_i$,   $i=1$, $2$, $\ldots$,  $k$,   and  apply  $\square$  or  $*$  to these  two randomly generated  hypergraphs.    
    {\bf  (5.b)}.  Choose  a randomly  generated hypergraph  from the remaining  $k-2$  randomly generated  hypergraphs     $ \mathcal{H}_{i}$,   $i=1$, $2$, $\ldots$,  $k$,  and  apply  $\square$  or $*$  to  this    chosen   randomly  generated hypergraph  and  the  randomly  generated  hypergraph  given  in  (5.a).   
    {\bf  (5.c)}.  Repeat   (5.b)  for  $k-2$  times.   This  gives  a   randomly  generated  hypergraph  $\mathcal{H}$.

          {\bf  Step~6}.    Choose  $w\in  G^1$.   Apply  $w$  to  $\mathcal{H}$.   The  final  randomly   generated  
          hypergraph  is  $w(\mathcal{H})$.  
    \end{proof}

Some  algorithms   that  generate     random  simplicial complexes  and  random  independence  hypergraphs  follow  immediately.

            \begin{proof}[Algorithm generating  random simplicial complexes]
     {\bf  Step~1}.         Same  as   Step~1  in    the  algorithm generating  random hypergraphs.

    {\bf  Step~2}.  For each  $i=1,2,\ldots,  k$,   choose  a  positive  integer   $n_i$.  
      Use  the     algorithm generating  random hypergraphs   to  give  $n_i$  randomly generated   hypergraphs  $\mathcal{H}_{1,i}$,  $\mathcal{H}_{2,i}$,  $\ldots$,  $\mathcal{H}_{n_i,i}$   on  $V_i$.

    {\bf  Step~3}.  For each  $i=1,2,\ldots,  k$  and  each  $j=1,2,\ldots,  n_i$,    apply  $\Delta$  or  $\delta$ 
    to  $\mathcal{H}_{j,i}$.  This  gives  a  randomly  generated  simplicial  complex  $\mathcal{K}_{j,i}$  on  $V_i$.

    {\bf  Step~4}.   For each  $i=1,2,\ldots,  k$,    similar  to  Step~4  of  the  previous algorithm,  
    apply   the binary operations  $\cap$  and   $\cup$   to  $\mathcal{K}_{1,i}$,  $\mathcal{K}_{2,i}$,   $\ldots$,  
    $\mathcal{K}_{n_i,i}$.    After $n_i-1$ times  of  the  binary  operations,     we  obtain   a  randomly  generated  simplicial  complex  $\mathcal{K}_i$  on  $V_i$.

             {\bf  Step~5}.   Similar  to  Step~5  of  the  previous algorithm,  apply  the  binary  operations  $*$,  $\Delta\square$  and  
             $\delta\square$  to  $\mathcal{K}_1$,  $\mathcal{K}_2$,  $\ldots$,  $\mathcal{K}_k$.  
              After $k-1$ times  of  the  binary  operations,     we  obtain   a  randomly  generated  simplicial  complex  $\mathcal{K}$.               
\end{proof}

 \begin{proof}[Algorithm generating  random independence  hypergraphs]
     {\bf  Steps~1-2}.  Same  as      Steps~1-2   in  the    algorithm generating  random  simplicial complexes.

    {\bf  Step~3}.  For each  $i=1,2,\ldots,  k$  and  each  $j=1,2,\ldots,  n_i$,    apply  $\bar\Delta$  or  $\bar\delta$ 
    to  $\mathcal{H}_{j,i}$ with respect to  $V_i$.  This  gives  a  randomly  generated  independence  hypergraph   $\mathcal{L}_{j,i}$  on  $V_i$.

    {\bf  Step~4}.   For each  $i=1,2,\ldots,  k$,    similar  to  Step~4  of  the  previous algorithms,  
    apply   the binary operations  $\cap$  and   $\cup$   to  $\mathcal{L}_{1,i}$,  $\mathcal{L}_{2,i}$,   $\ldots$,  
    $\mathcal{L}_{n_i,i}$.    After $n_i-1$ times  of  the  binary  operations,     we  obtain   a  randomly  generated  independence  hypergraph  $\mathcal{L}_i$  on  $V_i$.

             {\bf  Step~5}.   Similar  to  Step~5  of  the  previous algorithms,  apply  the  binary  operations  $*$,  $\bar\Delta\square$  and  
             $\bar\delta\square$  to  $\mathcal{L}_1$,  $\mathcal{L}_2$,  $\ldots$,  $\mathcal{L}_k$.  
              After $k-1$ times  of  the  binary  operations,     we  obtain   a  randomly  generated  independence  hypergraph  $\mathcal{L}$.               
\end{proof}

\smallskip

\noindent {{\bf{Acknowledgement}}. The present author would like to express his deep gratitude to  the  editor  and  the referee   
for their kind helps.  }

    {\footnotesize

    \bigskip

Shiquan Ren

Address:
School  of  Mathematics and Statistics,  Henan University,  Kaifeng   475004,  China.

e-mail:  renshiquan@henu.edu.cn

    %Let   $n$  be  a  positive  integer.   Let  $[n]$  be  the  vertex  set   $\{0,1,\ldots, n-1\}$.  
    %Let  $S(n)$   be  the  subset  of  $S$  such that  $V=[n]$.     Similarly,  we  use  the  notations  
    %$T(n)$,  $T_\gamma(n)$,   $R(n)$   and    $R_\gamma(n)$.   
    %  Suppose 
   % $n$  is  a  finite  sum  of  finite  products  of  positive  integers 
    %\begin{eqnarray*}
   % n=  \sum_i    \prod_j   n_{i,j}.  
    %\end{eqnarray*}
   % Then  
   % \begin{eqnarray*}
   %\ast_i \square_j   G_0(S(n_{i,j})  \cup T(n_{i,j})  \cup  T_\gamma(n_{i,j})   \cup  R(n_{i,j}) \cup   R_\gamma (n_{i,j}) ) 
 %   \end{eqnarray*}
 %   gives   a  family  of   random  hypergraphs    on   $V$.  
    
   % \begin{eqnarray*}
    %S\cup  T  \cup  T_\gamma  \cup  R  \cup  R_\gamma  \subseteq   G(S).  
   % \end{eqnarray*}
%Similarly,    we  use  the notations  $G(T)$,  $G(T_\gamma)$,   $G(R)$  and  $G(R_\gamma)$.  

    }

 \end{document}